\documentclass[reqno,tbtags,a4paper]{amsart}

\usepackage{amsmath,amssymb,amsthm,amsfonts,amscd,array,graphicx}

\newcommand{\bA}{\boldsymbol{A}}
\newcommand{\bS}{\boldsymbol{S}}
\newcommand{\bs}{\boldsymbol{s}}
\newcommand{\bell}{\boldsymbol{\ell}}

\newcommand{\V}{I}
\newcommand{\fV}{\mathfrak{V}}

\newcommand{\tR}{\widetilde{R}}

\newcommand{\F}{\mathbb{F}}

\newcommand{\R}{\mathbb{R}}
\newcommand{\C}{\mathbb{C}}
\newcommand{\Q}{\mathbb{Q}}

\newtheorem{theorem}{Theorem}[section]

\newtheorem{cor}[theorem]{Corollary}
\newtheorem{lem}[theorem]{Lemma}

\theoremstyle{definition}

\author{Alexander A. Gaifullin}

\thanks{The work was partially supported by the Russian Foundation for Basic Research (projects 12-01-31444 and 13-01-12469), by grants of the President of the Russian Federation (projects MD-4458.2012.1, NSh-4995.2012.1), by a grant of the Government of the Russian Federation (project 11.G34.31.0053), and by a grant from Dmitry Zimin's ``Dynasty'' foundation.}

\title[Volume of a simplex as a  function of the areas of its two-faces]{Volume of a simplex as a multivalued algebraic function of the areas of its two-faces}

\date{}

\address{Steklov Mathematical Institute, Moscow, Russia\newline
${}$\hspace{4.3mm}Kharkevich Institute for Information Transmission Problems, Moscow, Russia\newline 
${}$\hspace{4.3mm}Demidov Yaroslavl State University, Yaroslavl, Russia}

\email{agaif@mi.ras.ru}

\begin{document}

\begin{abstract}
 For $n\ge 4$, the square of the volume of an $n$-simplex satisfies a polynomial relation with coefficients depending on the squares of the areas of $2$-faces of this simplex. First, we compute the minimal degree of such polynomial relation. Second, we prove that the volume an $n$-simplex satisfies a monic polynomial relation with coefficients depending on the areas of $2$-faces of this simplex if and only if $n$ is even and $n\ge 6$, and we study the leading coefficients of polynomial relations satisfied by the volume for other~$n$.
\end{abstract}

\maketitle

\section{Introduction}

Let $\Delta^n$ be an $n$-dimensional simplex with vertices~$p_0,\ldots,p_n$ in the Euclidean space~$\R^n$. The classical Cayley--Menger formula allows us to compute the volume~$V$ of~$\Delta^n$ from the lengths~$\ell_{ij}=|p_ip_j|$ of edges of~$\Delta^n$, cf.~\cite[section~40]{Blu70}:

\begin{equation}\label{eq_CM}
V^2=\frac{(-1)^{n+1}}{2^n(n!)^2}\cdot
\det\left(
\begin{matrix}
0 & 1 & 1 & 1 & \cdots & 1\\
1 & 0 & \ell_{01}^2 & \ell_{02}^2 & \cdots & \ell_{0n}^2\\
1 & \ell_{01}^2 & 0 & \ell_{12}^2 & \cdots & \ell_{1n}^2\\
1 & \ell_{02}^2 & \ell_{12}^2 & 0 & \cdots & \ell_{2n}^2\\
\vdots & \vdots & \vdots & \vdots & \ddots & \vdots\\
1 & \ell_{0n}^2 & \ell_{1n}^2 & \ell_{2n}^2 & \cdots & 0
\end{matrix}
\right)
\end{equation} 
The lengths of edges of an $n$-simplex~$\Delta^n$ are algebraically independent. We denote by~$\bell$ the set of ${n+1\choose2}$ independent variables~$\ell_{ij}=\ell_{ji}$,  and denote by~$\Q[\bell]$ and~$\Q(\bell)$ the ring of polynomials and the field of rational functions in these variables with coefficients in~$\Q$ respectively. Let~$\F$  be the algebraic closure of~$\Q(\bell)$.

Denote by $A_{i_0\ldots i_k}$ the $k$-volume of the $k$-face of~$\Delta^n$ spanned by the vertices $p_{i_0},\ldots,p_{i_k}$.  Denote by~$\bA_{k}$ the set of ${n+1}\choose{k+1}$ volumes $A_{i_0\ldots i_k}$ of $k$-faces of~$\Delta^n$. 

Our goal is to study  the volume~$V$ of a simplex as a multivalued function of the $k$-volumes~$A_{i_0\ldots i_k}$ of its $k$-faces. We shall mostly focus on the case $k=2$. 

The first natural question is whether the volume of an $n$-simplex is algebraic or not over the field~$L_k$ generated by the $k$-volumes of all its $k$-faces, i.\,e., whether there exists an algebraic relation
\begin{equation}\label{eq_main_alg}
a_0(\bA_{k})V^N+a_1(\bA_{k})V^{N-1}+\cdots+a_N(\bA_{k})=0
\end{equation}
that holds for all $n$-simplices~$\Delta^n$, where $a_i(\bA_{k})$ are polynomials  in the $k$-volumes $A_{i_0\ldots i_k}$ such that $a_0(\bA_{k})$ is not  identically zero for all~$\Delta^n$.  The precise definition of the field~$L_k$ is as follows. By the Cayley--Menger formula~\eqref{eq_CM}, the squares $W=V^2$ and $S_{i_0\ldots i_k}=A^2_{i_0\ldots i_k}$ are polynomials in~$\ell_{ij}$, i.\,e., elements of~$\Q[\bell]$. Then for every~$k$, $L_k$ is the subfield of~$\F$ generated over~$\Q$ by all elements~$A_{i_0\ldots i_k}=\sqrt{S_{i_0\ldots i_k}}$, and we are interested whether the element~$V=\sqrt{W}$ is algebraic or not over~$L_k$.   

It is not hard to show that $V$ is not algebraic over~$L_{n-1}$ for $n\ge 3$. Mohar and Rivin~\cite{MoRi10} showed that for $n\ge 2$, the ${n+1\choose 2}$ volumes~$A_{i_0\ldots i_{n-2}}$ of $(n-2)$-faces of~$\Delta^n$ are algebraically independent. Hence, $\Q(\bell)$ and $L_{n-2}$ have equal transcendancy degrees over~$\Q$. Since all elements of~$L_{n-2}$ are algebraic over~$\Q(\bell)$, this implies that the algebraic closures  of~$\Q(\bell)$ and of~$L_{n-2}$ coincide. Hence, $V$ is algebraic over~$L_{n-2}$.  Now, suppose that $k<n-2$. The result of Mohar and Rivin implies that the edge lengths of any $(k+2)$-face~$F\subset\Delta^n$ are algebraic over the field generated the $k$-volumes of $k$-faces of~$F$. Therefore, all edge lengths of~$\Delta^n$ are algebraic over~$L_k$, hence, $V$ is algebraic over~$L_k$.

Instead of relations~\eqref{eq_main_alg}, we shall study relations of the form
\begin{equation}\label{eq_bW}
b_0(\bS_{k})W^M+b_1(\bS_{k})W^{M-1}+\cdots+b_M(\bS_{k})=0,
\end{equation}
where $\bS_k$ denotes the set of ${n+1\choose k+1}$ squares $S_{i_0\ldots i_k}=A_{i_0\ldots i_k}^2$. This is more natural, since, unlike~$A_{i_0\ldots i_k}$ and~$V$, their squares~$S_{i_0\ldots i_k}$ and~$W$ are polynomials in the edge lengths. In other words, we shall study polynomials satisfied by~$W$ over the field~$K_k$ generated over~$\Q$ by the ${n+1\choose k+1}$ elements~$S_{i_0\ldots i_k}$. Our first goal is to compute the minimal degree of such polynomial for $k=2$.

\begin{theorem}\label{theorem_degree_V}
If\/ $n\ge 6$ and $n$ is even, then $W$ is a rational function in the squares of the areas of\/ $2$-faces, i.\,e., $W\in K_2$. If\/ $n\ge 5$ and $n$ is odd, then $W^2$ is a rational function in the squares of the areas of\/ $2$-faces, i.\,e., $W^2\in K_2$, but $W\notin K_2$. If\/ $n=4$, then the degree of the minimal polynomial of\/ $W$ over~$K_2$ is equal to~$32$.     
\end{theorem} 

A similar result can be obtained for the squares of the edge lengths~$s_{ij}=\ell_{ij}^2$. 

\begin{theorem}\label{theorem_degree_edge}
If\/ $n\ge 5$, then any element $s_{ij}s_{kl}$ \textnormal{(}in particular, any element~$s_{ij}^2$\textnormal{)} is a rational function in the squares of the areas of\/ $2$-faces, i.\,e., belongs to~$K_2$, but none of~$s_{ij}$ belongs to~$K_2$. If\/ $n=4$, then the degree of the minimal polynomial of\/~$s_{ij}$ over~$K_2$ is equal to~$64$.
\end{theorem}

\begin{cor}\label{cor_generic}
For a generic simplex $\Delta^n\subset \R^n$, $n\ge 5$, any simplex with the same areas of the corresponding $2$-faces is congruent to it.
\end{cor}

{\sloppy
\begin{theorem}\label{theorem_number}
For a generic set  of~$10$  positive real numbers~$A_{ijk}$, where $0\le i<j<k\le 4$, there exist not more than $32$ congruence classes of\/ $4$-simplices with the $2$-face areas~$A_{ijk}$. 
\end{theorem}
}

The next natural question is whether the volume~$V$ is integral or not over the $\Q$-subalgebra $R_k$ of~$\F$ generated by all $k$-volumes $A_{i_0\ldots i_k}$, i.\,e., whether there exists a \textit{monic\/} polynomial relation
\begin{equation*}%\label{eq_main_int}
V^N+a_1(\bA_{k})V^{N-1}+\cdots+a_N(\bA_{k})=0
\end{equation*}
that holds for all $n$-simplices~$\Delta^n$, where $a_i(\bA_{k})$ are polynomials  in~$A_{i_0\ldots i_k}$.

One of the motivations for this question is a result by Sabitov~\cite{Sab96} that the volume of any simplicial polyhedron in~$\R^3$ satisfies a monic polynomial relation with coefficients depending on the edge lengths of the polyhedron, cf.~\cite{CSW97},~\cite{Sab98}. This result has been generalized to higher dimensions by the author~\cite{Gai11},~\cite{Gai12}. A natural question is whether it is possible or not to replace the edge lengths by the areas of $2$-faces, cf.~\cite{Sab11}.

\begin{theorem}\label{theorem_R}
The element $V$ is integral over the ring~$R_2$ if and only if $n$ is even and $n\ge 6$. 
\end{theorem}

\begin{cor}\label{cor_Rk}
Suppose that $n$ is odd, $k$ is even, $0<k<n$, and $k\ne 4$. Then the element~$V$ is not integral over the ring~$R_k$. 
\end{cor}

\begin{proof}
If $k=2$, the corollary coincides with Theorem~\ref{theorem_R}. Suppose, $k\ge 6$. Applying Theorem~\ref{theorem_R} to $k$-dimensional faces of~$\Delta^n$, we obtain that all their $k$-volumes are integral over~$R_2$. Hence, all elements of~$R_k$ are integral over~$R_2$. By Theorem~\ref{theorem_R}, $V$ is not integral over~$R_2$. Therefore, $V$ is not integral over~$R_k$. 
\end{proof}

A classical question is whether the $k$-volumes of $k$-faces of an $n$-simplex determine its congruence class uniquely, or at least up to finitely many possibilities. Corollary~\ref{cor_generic} and Theorem~\ref{theorem_number} answer this question in \textit{generic\/} situation. However, it is much more interesting to obtain results that hold for \textit{all\/} simplices. The most interesting case is $k=n-2$, since the congruence class of an $n$-simplex depends on ${n+1\choose 2}$ parameters, which is equal to the number of $(n-2)$-dimensional faces. The question due to Connelly is whether the $(n-2)$-volumes of $(n-2)$-faces of an $n$-simplex determine uniquely the \textit{volume\/} of it. It is also natural to ask for a given $k\le n-2$, whether the volume of an $n$-simplex is determined by the $k$-volumes of its  $k$-faces up to finitely many possibilities. This motivates the following algebraic question. \textit{Does there exists a relation of the form~\eqref{eq_main_alg} such that $a_0(\bA_k)$ is non-zero whenever we substitute for~$\bA_k$ the set of the $k$-volumes of $k$-faces of a non-degenerate~$n$-simplex?\/}  Equivalently,  let $\tR_k$ be the $\Q$-algebra consisting of all rational functions~$\frac{a(\bA_k)}{b(\bA_k)}$  such that $b(\bA_k)$ is non-zero whenever we substitute for~$\bA_k$ the set of the $k$-volumes of $k$-faces of a non-degenerate~$n$-simplex. Then \textit{is $V$ integral over\/}~$\tR_k$?

The ring~$\tR_k$ is very complicated from the algebraic viewpoint. Hence it is reasonable to consider intermediate rings that contain~$R_k$ and are contained in~$\tR_k$. For such ring we shall take the $\Q$-subalgebra $\Lambda_k$ of~$\F$ generated by the $2{{n+1}\choose{k+1}}$ elements~$A_{i_0\ldots i_k}^{\pm 1}$, that is, the algebra  of Laurent polynomials in the elements~$A_{i_0\ldots i_k}$ (which may be algebraically dependent).

\begin{theorem}\label{theorem_Lambda}
The element $V$ is integral over~$\Lambda_2$ if and only if $n\ge 6$. Hence, for $n\ge 6$, $V$ is integral over~$\tR_2$.
\end{theorem}

\begin{cor}
For $n\ge 6$, the volume of an $n$-simplex with the given set of areas of its $2$-faces can take only finitely many values.
\end{cor}

Throughout this paper, we shall denote by~$I$ the set numerating the vertices of~$\Delta^n$, i.\,e., the set $\{0,1,\ldots,n\}$. (For abuse of notation, we do not indicate explicitly the dependance of all objects under consideration on the dimension~$n$.) 

When $V$ is not integral over~$R_2$, i.\,e., for $n=4$ or $n$ odd, it is interesting to describe the ideal $\mathfrak{a}\lhd R_2$ consisting of all~$a$ such that $aV$ is integral over~$R_2$. Again, it is natural to replace the elements~$V$ and~$A_{ijk}$ with their squares~$W$ and~$S_{ijk}$ respectively. Let $R_2'$ be the $\Q$-subalgebra of~$\F$ generated by ${n+1\choose 3}$ elements~$S_{ijk}$. Our goal is to decribe the ideal $\mathfrak{b}\lhd R_2'$ consisting of all~$b$ such that $bW$ is integral over~$R_2'$. (Obviously, an element is integral over~$R'_2$ if and only if it is integral over~$R_2$.) The best result can be obtain for~$n=4$. In this case $R_2'=\Q[\bS_2]$ is the polynomial algebra in~$10$ independent variables~$S_{ijk}$. 

\begin{theorem}\label{theorem_leading}
For $n=4$, the ideal $\mathfrak{b}\lhd\Q[\bS_2]$ consisting of all elements $b$ such that $bW$ is integral over~$\Q[\bS_2]$ is the principal ideal generated by the element
\begin{multline}\label{eq_Q}
Q=\prod_{0\le i<j\le 4}(2S_{ijk}S_{ijl}+2S_{ijl}S_{ijm}+2S_{ijm}S_{ijk}-S_{ijk}^2-S_{ijl}^2-S_{ijm}^2)\\
=\smash{\prod_{0\le i<j\le 4}}(A_{ijk}+A_{ijl}+A_{ijm})(-A_{ijk}+A_{ijl}+A_{ijm})(A_{ijk}-A_{ijl}+A_{ijm})\\(A_{ijk}+A_{ijl}-A_{ijm}),
\end{multline}
where in each multiplier, $k$, $l$, $m$ are the three elements of the set $I=\{0,1,2,3,4\}$ distinct from~$i$ and~$j$.
\end{theorem}

\begin{cor}
For $n=4$, the leading coefficient of any polynomial relation of the form~\eqref{eq_bW} is divisible by~$Q$. On the other hand, there exists a polynomial relation  of the form~\eqref{eq_bW} with the leading coefficient $Q^d$ for some positive integer~$d$.
\end{cor}

For odd $n$ we can obtain the following partial result. For any pairwise distinct $i_1,i_2,i_3,i_4,i_5\in I$, we put
\begin{equation}\label{eq_P}
P_{i_1i_2i_3i_4i_5}=\prod_{1\le p<q<r\le 5}S_{i_pi_qi_r}.
\end{equation}
If $n=5$, then we also put
\begin{equation}\label{eq_D}
D=\prod (S_{ijk}-S_{ijl}),
\end{equation}
where the product is taken over all pairwise distinct~$i,j,k,l\in I=\{0,1,2,3,4,5\}$ such that $i<j$ and~$k<l$.

\begin{theorem}\label{theorem_leading2}
If $n$ is odd and $n\ge 7$, then $P_{i_1i_2i_3i_4i_5}\in\mathfrak{b}$, that is, the element $P_{i_1i_2i_3i_4i_5}W$ is integral over~$R_2'$ for any pairwise distinct $i_1,i_2,i_3,i_4,i_5\in I$. If $n=5$, then $P_{i_1i_2i_3i_4i_5}D\in\mathfrak{b}$, that is, the element~$P_{i_1i_2i_3i_4i_5}DW$ is integral over~$R_2'$ for any pairwise distinct $i_1,i_2,i_3,i_4,i_5\in I$.
\end{theorem}

The paper is organized as follows. In section~\ref{section_degree} we prove Theorems~\ref{theorem_degree_V}, \ref{theorem_degree_edge} and~\ref{theorem_number}. In section~\ref{section_negative} we prove the ``negative'' parts of Theorems~\ref{theorem_R} and~\ref{theorem_Lambda}, namely, we prove that $V$ is not integral over~$R_2$ if $n$ is odd, and $V$ is not integral over~$\Lambda_2$ if $n=4$ or~$5$. The proofs of ``positive'' parts of  Theorems~\ref{theorem_R} and~\ref{theorem_Lambda} use theory of \textit{valuation of fields\/} or, equivalently, \textit{places of fields\/}. Places of fields were first used in the metric theory of polyhedra by Connelly, Sabitov, and Walz~\cite{CSW97} to obtain another proof of Sabitov's theorem on the volume.  Then they were used by the author~\cite{Gai11},~\cite{Gai12} to generalize Sabitov's theorem to higher dimensions, and by S.\,A.\,Gaifullin and the author~\cite{GaGa13} to study flexible periodic polyhedral surfaces. In our situation the language of \textit{valuations\/} turns out to be more convenient than the language of \textit{places\/}. In sections~\ref{section_valuations}--\ref{section_typeII} we study valuations of the field~$\F$. Then in section~\ref{section_final} we apply the obtained results to prove Theorems~\ref{theorem_R}, \ref{theorem_Lambda}, \ref{theorem_leading} and~\ref{theorem_leading2}.

The author is grateful to S.\,A.\,Gaifullin and I.\,Kh.\,Sabitov for useful discussions.

\section{Degrees of minimal polynomials}\label{section_degree}

Recall that $K_1$ is the subfield of~$\Q(\bell)$ generated over~$\Q$ by all~$s_{ij}=\ell_{ij}^2$, and  $K_2$ is the subfield of~$\Q(\bell)$ generated over~$\Q$ by all~$S_{ijk}=A_{ijk}^2$. We have $K_2\subset K_1$, since by Heron's formula, 
\begin{equation}\label{eq_reg_map}
S_{ijk}=\frac{1}{16}(2s_{ij}s_{jk}+2s_{jk}s_{ki}+2s_{ki}s_{ij}-s_{ij}^2-s_{jk}^2-s_{ki}^2),
\end{equation}
 
\begin{theorem}\label{theorem_degree}
The degree of~$K_1$ over~$K_2$ is equal to~$64$ if\/ $n=4$, and is equal to~$2$ if\/ $n>4$.
\end{theorem}

Consider the affine space~$X=\C^{n+1\choose 2}$ with the coordinates $s_{ij}=s_{ji}$, and the affine space~$\C^{n+1\choose 3}$ with the coordinates~$S_{ijk}$, where the order of  $i,j,k$ is irrelevant. Formula~\eqref{eq_reg_map} yields the regular morphism $\varphi\colon X\to\C^{n+1\choose 3}$. Let $Y$ be the Zariski closure of its image. The fields of rational functions $\Q(X)$ and $\Q(Y)$ are naturally isomorphic to $K_1$ and $K_2$ respectively. Hence the degree~$|K_1/K_2|$ of $K_1$ over~$K_2$ is equal to the degree of the mapping $\varphi\colon X\to Y$, that is, to the  number of pre-images of a regular value of~$\varphi$. A point $y\in Y$ is a regular value of~$\varphi$ if

(1) $y$ is a smooth point of~$Y$,

(2) $y$ has a neighbourhood~$U$ (in the analytic topology) such that $\varphi^{-1}(U)$ is bounded in $X=\C^{n+1\choose 2}$,

(3) $y$ has finitely many pre-images~$x_1,\ldots,x_d$ under~$\varphi$ at each of which the differential $D\varphi\colon T_{x_i}X\to T_yY$ is bijective. 

Let $y_0\in\C^{n+1\choose 3}$ be the point with all coordinates~$S_{ijk}$ equal to~$3/16$, that is, to the square of the area of a regular triangle with edge~$1$. We have $y_0\in Y$, since $y_0=\varphi(x_0)$ for the point $x_0\in X$ with all coordinates $s_{ij}$ equal to~$1$. Let us describe all pre-images of~$y_0$ under the mapping~$\varphi$.

A simplex $\Delta^n\subset\R^n$ is called \textit{equiareal\/} if the areas of all its $2$-faces are equal to each other. McMullen~\cite{McM00} proved that for $n\ge 4$, any equiareal simplex is regular. This means that the point~$x_0$ is the only pre-image of~$y_0$ such that the set of its coordinates~$s_{ij}$ can be realised as the set of the squares of edge lengths of a simplex $\Delta^n\subset\R^n$. We are interested in \textit{all\/}, realisable or not, pre-images of~$y_0$ under~$\varphi$. This  is very close to the problem of describing all \textit{complex\/} equiareal simplices $\Delta^n\subset\C^n$. Here, $\C^n$ is endowed with the standard bilinear (not Hermitian) inner product, the squares of edge lengths of~$\Delta^n$ are defined by means of this inner product, and the squares of areas of $2$-faces are defined by~\eqref{eq_reg_map}.  (By definition, a simplex $\Delta^n\subset\C^n$ is the convex hull of $n+1$ points affinely independent over~$\C$.) McMullen's technique is still useful in this setting. However, we face some principally new phenomena caused by the existence of isotropy vectors, and the answer is non-trivial.  

A \textit{partial pairing\/}~$\omega$ on~$I$ is a set of non-intersecting two-element subsets of~$I$. The set~$I$ consists of $n+1$ elements, hence, the number of partial pairings on~$I$ is 
$$
r_n=\sum_{k=0}^{[\frac{n+1}{2}]}\frac{(n+1)!}{2^k\cdot k!(n+1-2k)!}.
$$
Let $\omega$ be a partial pairing on~$I$, and let $\sigma$ be either $1$ or $-1$. Consider a point $x_{\omega,\sigma}\in  X$ with coordinates $s_{ij}=\sigma$ for $\{i,j\}\notin\omega$ and $s_{ij}=3\sigma$ for $\{i,j\}\in\omega$. It can be easily checked that $\varphi(x_{\omega,\sigma})=y_0$. Obviously, $x_{\emptyset,1}=x_0$.

Now, suppose $n=4$. Let $\gamma$ be a set of two-element subsets of $I=\{0,1,2,3,4\}$ that form a $5$-cycle. This means that $$\gamma=\{\{i_0,i_1\},\{i_1,i_2\},\{i_2,i_3\},\{i_3,i_4\},\{i_4,i_0\}\}$$ for a permutation $i_0,i_1,i_2,i_3,i_4$ of $0,1,2,3,4$. Notice that all two-element subsets of~$I$ not belonging to~$\gamma$ also form a $5$-cycle. Consider a point $x_{\gamma}\in  X$ with coordinates $s_{ij}=\sqrt{-3/5}$ for $\{i,j\}\in\gamma$ and $s_{ij}=-\sqrt{-3/5}$ for $\{i,j\}\notin\gamma$. It is easy to see that $\varphi(x_{\gamma})=y_0$.
The number of $5$-cycles~$\gamma$ in~$I$ is equal to~$12$.

\begin{lem}\label{lem_pre-images}
For any~$n\ge 5$, the point $y_0$ has exactly $2r_n$ pre-images under~$\varphi$, namely, the points~$x_{\omega,\sigma}$. For $n=4$, the point $y_0$ has exactly $64$ pre-images under~$\varphi$, namely, the $2r_4=52$ points~$x_{\omega,\sigma}$ and the $12$ points~$x_{\gamma}$. Moreover,  for all sufficiently small $\varepsilon>0$, if a point~$y\in Y$ is coordinatewise  $\varepsilon$-close to~$y_0$, then any pre-image of~$y$  is $O(\sqrt{\varepsilon})$-close to one of the pre-images of~$y_0$. 
\end{lem}

We start with the description of complex equiareal tetrahedra $\Delta^3\subset\C^3$. In the real case it is well known that each pair of opposite edges of an equiareal  tetrahedron have the same lengths. In complex case, our exposition mimics the proof of this fact due to McMullen~\cite[Theorem~4.2]{McM00}, though the answer is different.

\begin{lem}\label{lem_tetrahedron}
Let $\Delta^3\subset\C^3$ be a tetrahedron with vertices $p_0,p_1,p_2,p_3$ and squares of edge lengths $s_{ij}$ such that the squares of areas of all its $2$-faces are equal to~$3/16$. Then  either $s_{12}=s_{03}$, $s_{13}=s_{02}$, and $s_{23}=s_{01}$ or there is a cyclic permutation $i,j,k$ of\/ $1,2,3$ such that $s_{ij}=s_{0k}$, $s_{ik}=s_{0j}$, and $s_{jk}+s_{0i}-2s_{0j}-2s_{0k}=0$.
 \end{lem}

\begin{proof}
We put $\xi_i=p_i-p_0$, $i=1,2,3$. We denote by $(\cdot,\cdot)$, $[\cdot,\cdot]$, and $\langle\cdot,\cdot,\cdot\rangle$ the standard bilinear inner product, the standard vector product, and the standard mixed product in~$\C^3$ respectively. Then $\eta_1=[\xi_2,\xi_3]$, $\eta_2=[\xi_3,\xi_1]$, $\eta_3=[\xi_1,\xi_2]$, and $\eta_0=[\xi_3-\xi_1,\xi_2-\xi_1]$ are the normal vectors to the faces of~$\Delta^3$ opposite to the vertices $p_1$, $p_2$, $p_3$, and~$p_0$ respectively such that $(\eta_i,\eta_i)=3/8$. Obviously, $\eta_0+\eta_1+\eta_2+\eta_3=0$. Besides, we have
\begin{equation}\label{eq_vect_prod}
\xi_1=\frac{[\eta_2,\eta_3]}{\langle\xi_1,\xi_2,\xi_3\rangle}\,,\quad
\xi_2=\frac{[\eta_3,\eta_1]}{\langle\xi_1,\xi_2,\xi_3\rangle}\,,\quad
\xi_3=\frac{[\eta_1,\eta_2]}{\langle\xi_1,\xi_2,\xi_3\rangle}\,.
\end{equation}
We put $\zeta_i=\eta_0+\eta_i$, $i=1,2,3$. Let us prove that at least one of the three vectors $\zeta_1$, $\zeta_2$, and~$\zeta_3$ is not isotropic. Assume the converse, i.\,e., $(\zeta_i,\zeta_i)=0$, $i=1,2,3$. Then the Gram matrix of the vectors~$\eta_0,\eta_1,\eta_2,\eta_3$ is equal to
\begin{equation}\label{eq_G}
G=\frac{3}{8}\left(\begin{array}{rrrr}
1&-1&-1&-1\\
-1&1&-1&-1\\
-1&-1&1&-1\\
-1&-1&-1&1
\end{array}\right).
\end{equation}
This matrix is non-degenerate, which is impossible, since the vectors~$\eta_0,\eta_1,\eta_2,\eta_3$ are linearly dependent. Hence, there is an~$i$ such that $(\zeta_i,\zeta_i)\ne 0$. Let $i,j,k$ be the cyclic permutation of $1,2,3$ starting from~$i$. Let $\rho$ be the rotation by~$\pi$ around the vector~$\zeta_i$, i.\,e., the linear operator given by
\begin{equation}\label{eq_rho}
\rho(\theta)=\frac{2(\theta,\zeta_i)}{(\zeta_i,\zeta_i)}\zeta_i-\theta.
\end{equation}
Then $\rho\in SO(3,\C)$ and $\rho^2=\mathrm{id}$. It is easy to check that $\rho$ takes the vectors~$\eta_0$, $\eta_i$, $\eta_j$, and~$\eta_k$ to the vectors~$\eta_i$, $\eta_0$, $\eta_k$, and~$\eta_j$ respectively. Using~\eqref{eq_vect_prod}, we obtain that
\begin{equation}\label{eq_rho_xi}
\rho(\xi_j)=\xi_k-\xi_i,\qquad\rho(\xi_k)=\xi_j-\xi_i.
\end{equation}
Therefore,
$
s_{ik}=(\xi_k-\xi_i,\xi_k-\xi_i)=(\xi_j,\xi_j)=s_{0j},
$
and, similarly, $s_{ij}=s_{0k}$.
Since $s_{ik}=s_{0j}$ and $s_{ij}=s_{0k}$, the equality $S_{0ij}=S_{ijk}$ can be rewritten as
\begin{equation}\label{eq_product_zero}
(s_{jk}-s_{0i})(s_{jk}+s_{0i}-2s_{0j}-2s_{0k})=0.
\end{equation}
Hence either $s_{jk}=s_{0i}$ or $s_{jk}+s_{0i}-2s_{0j}-2s_{0k}=0$.
\end{proof}

%Indeed, we need to prove the following ``quantitative'' analogue of Lemma~\ref{lem_tetrahedron}.

\begin{lem}\label{lem_tetrahedron_est}
Let $s_{ij}=s_{ji}$, $0\le i,j\le 3$, $i\ne j$, be complex numbers, $M=\max|s_{ij}|$, and let  $S_{ijk}$ be given by~\eqref{eq_reg_map}. Then there exist universal positive constants~$E$ and~$C$ such that the following estimate holds. If\/ $|S_{ijk}-3/16|<\varepsilon<E$ for all pairwise distinct $i,j,k$, then there is a cyclic permutation $i,j,k$ of\/~$1,2,3$ such that $|s_{ij}-s_{0k}|<CM\varepsilon$, $|s_{ik}-s_{0j}|<CM\varepsilon$, and either $|s_{jk}-s_{0i}|<CM\sqrt{\varepsilon}$ or $|s_{jk}+s_{0i}-2s_{0j}-2s_{0k}|<CM\sqrt{\varepsilon}$.
 \end{lem}

\begin{proof}
It is a standard fact from theory of matrices that any non-degenerate symmetric complex matrix is the square of a symmetric matrix. Hence any non-degenerate symmetric complex matrix of order~$n$ can be realised as the Gram matrix of a basis in~$\C^n$  with respect to the standard bilinear inner product. This easily yields that a $6$-tuple $(s_{01},\ldots,s_{23})$ of complex numbers can be realised as the squares of the edge lengths of a tetrahedron $\Delta^3\subset\C^3$ if and only if the corresponding Cayley--Menger determinant~\eqref{eq_CM} is non-zero. In particular,  realisable $6$-tuples form a Zariski open subset of~$\C^6$. Therefore, it is sufficient to prove the lemma for realisable $6$-tuples $(s_{01},\ldots,s_{23})$.

The proof mimics the above prove of Lemma~\ref{lem_tetrahedron}. The vectors~$\xi_i$, $\eta_i$, and~$\zeta_i$ are defined by the same formulae. Instead of the equalities~$(\eta_i,\eta_i)=3/8$, we now have  estimates $|(\eta_i,\eta_i)-3/8|<2\varepsilon$.
If $\varepsilon$ and all three numbers~$|(\zeta_i,\zeta_i)|$ were sufficiently small, the Gram matrix of~$\eta_0$, $\eta_1$, $\eta_2$, and~$\eta_3$ would be close to the matrix~$G$ given by~\eqref{eq_G}, hence, non-degenerate, which is impossible. Therefore, there is a universal constant~$c_1>0$ such that for~$\varepsilon$ small enough, one of~$|(\zeta_i,\zeta_i)|$ is greater than~$c_1$. We define the rotation~$\rho$ by the same formula~\eqref{eq_rho}. Instead of~\eqref{eq_rho_xi}, we obtain that 
\begin{equation*}%\label{eq_rho_xi}
\rho(\xi_j)=\xi_k-\xi_i+\lambda_1\xi_1+\lambda_2\xi_2+\lambda_3\xi_3,\qquad
\rho(\xi_k)=\xi_j-\xi_i+\mu_1\xi_1+\mu_2\xi_2+\mu_3\xi_3,
\end{equation*}
where $|\lambda_i|,|\mu_i|<c_2\varepsilon$ for a universal 
 constant~$c_2>0$. This yields $|s_{ik}-s_{0j}|<c_3M\varepsilon$, $|s_{ij}-s_{0k}|<c_3M\varepsilon$ for a universal constant~$c_3>0$. Instead of~\eqref{eq_product_zero}, we obtain
$$
|(s_{jk}-s_{0i})(s_{jk}+s_{0i}-2s_{0j}-2s_{0k})|<(14c_3M^2+32)\varepsilon,
$$
which does not exceed~$c_4M^2\varepsilon$ for a universal constant~$c_4>0$, since $M$ cannot be arbitrarily close to zero. Therefore, either $|s_{jk}-s_{0i}|$ or $|s_{jk}+s_{0i}-2s_{0j}-2s_{0k}|$ does not exceed~$M\sqrt{c_4\varepsilon}$.
\end{proof}

\begin{proof}[Proof of Lemma~\ref{lem_pre-images}]
It is sufficient to prove the lemma for~$n=4$. Then the lemma for all $n>4$ will follow immediately.

Let $x$ be a pre-image of~$y_0$, let $s_{ij}$ be the coordinates of~$x$.
Lemma~\ref{lem_tetrahedron} implies that for any $4$-element subset $J\subset I$, there is a permutation $i,j,k,l$ of it such that $s_{ik}=s_{jl}$, $s_{jk}=s_{il}$, and either $s_{ij}=s_{kl}$ or $s_{ij}+s_{kl}-2s_{ik}-2s_{il}=0$. We shall apply this to different subsets~$J$. For $J=\{0,1,2,3\}$, we may assume that $s_{01}=s_{23}$ and $s_{02}=s_{13}$. We denote the numbers~$s_{01}$, $s_{02}$, and~$s_{03}$ by~$a$, $b$, and~$c$ respectively. For $J=\{0,2,3,4\}$, we obtain that at least one of the equalities $s_{04}=a$ and $s_{34}=b$ holds. We assume that $s_{04}=a$, since the other case is similar. Consider two cases:

First, suppose that $b\ne a$. For $J=\{0,1,3,4\}$, we obtain that $s_{34}=a$, $s_{14}=c$, and $a-b+2c=0$. Then for $J=\{0,2,3,4\}$, we obtain that  $s_{24}=c$. Further, for $J=\{0,1,2,4\}$, we obtain that $s_{12}=a$ and either $c=a$ or $c=b$. If $c=a$, then $b=3a$. Hence, all $s_{ij}$ are equal to~$a$, except for $s_{02}=s_{13}=3a$. Since $S_{012}=3/16$, we get $a=\pm 1$. Therefore, $x=x_{\omega,\pm 1}$, where $\omega=\{\{0,2\},\{1,3\}\}$. If $c=b$, then $b=-a$. We obtain that $s_{01}$, $s_{12}$, $s_{23}$, $s_{34}$, and~$s_{04}$ are equal to~$a$, and all other $s_{ij}$ are equal to~$-a$. Since $S_{012}=3/16$, we get $a=\pm\sqrt{-3/5}$. Therefore, $x=x_{\gamma}$, where $\gamma$ is either $\{\{0,1\},\{1,2\},\{2,3\},\{3,4\},\{4,0\}\}$ or $\{\{0,2\},\{2,4\},\{4,1\},\{1,3\},\{3,0\}\}$.

Second, suppose that $b=a$. Assume that $s_{34}\ne a$. Then for $J=\{0,2,3,4\}$, we get $s_{24}=c$ and $s_{34}=a+2c$. For $J=\{0,1,3,4\}$, we obtain that $s_{14}=c$. Now, for $J=\{1,2,3,4\}$, we obtain that $c=a$, hence, $s_{34}=3a$, and $s_{12}$ is either~$a$ or~$3a$. Since $S_{012}=3/16$, we get $a=\pm 1$. Therefore, $x$ coincides with one of the points~$x_{\omega,\sigma}$. Now assume that $s_{34}=a$. Then for $J=\{1,2,3,4\}$, we obtain that at least two of the three numbers~$s_{12}$, $s_{14}$, and~$s_{24}$ are equal to~$a$, and the third is either~$a$ or~$3a$. Also for $J=\{0,1,3,4\}$, we obtain that $s_{03}$ is either~$a$ or~$3a$. Again, we have $a=\pm 1$. Therefore, $x$ coincides with one of the points~$x_{\omega,\sigma}$.

Now, let $y\in Y$ be a point with coordinates $S_{ijk}$ such that $|S_{ijk}-3/16|<\varepsilon$, let $x\in X$ be a pre-image of~$y$ under~$\varphi$, let $s_{ij}$ be the coordinates of~$x$, and let $M=\max |s_{ij}|$. Using Lemma~\ref{lem_tetrahedron_est} instead of Lemma~\ref{lem_tetrahedron} in the above speculations, we can prove that the point $x$ is $O(M\sqrt{\varepsilon})$-close either to one of the points~$x_{\omega,\sigma}$ or to one of the points~$x_{\gamma}$. Hence, for sufficiently small~$\varepsilon$, all~$s_{ij}$ are bounded by a universal constant independent of~$\varepsilon$. Therefore, the estimate $const\cdot M\sqrt{\varepsilon}$ can be replaced with the required estimate~$const\cdot \sqrt{\varepsilon}$.
\end{proof}

\begin{lem}\label{lem_diff}
The differential $D(x)=D\varphi|_x\colon\C^{n+1\choose 2}\to\C^{n+1\choose 3}$ is injective at every point~$x=x_{\omega,\sigma}$ if $n\ge 4$, and at every point~$x=x_{\gamma}$ if $n=4$. 
\end{lem}

\begin{proof}
For $n=4$, a direct computation shows that the $10\times 10$ matrix $D(x)=\bigl(\frac{\partial S_{ijk}}{\partial s_{lm}}\bigr)$ is non-degenerate at every~$x_{\omega,\sigma}$ and at every~$x_{\gamma}$. For $n>4$, this implies that the $10\times10$ minor of every~$D(x_{\omega,\sigma})$ formed by rows and columns corresponding to $3$-element and $2$-element subsets  contained in every $5$-element subset~$J\subset I$ is non-zero. Besides, we have  $\frac{\partial S_{ijk}}{\partial s_{lm}}=0$ whenever $\{l,m\}\not\subset\{i,j,k\}$. It follows easily that every $D(x_{\omega,\sigma})$ has full rank.
\end{proof}

\begin{lem}\label{lem_diff2}
Suppose that $n>4$. Then the images of the differentials~$D(x_{\omega_1,\sigma_1})$ and~$D(x_{\omega_2,\sigma_2})$ coincide if and only if $\omega_1=\omega_2$.
\end{lem}

\begin{proof}
For $n=5$ the lemma can be checked by a direct computation. The general case follows from the case $n=5$.
\end{proof}

\begin{proof}[Proof of Theorem~\ref{theorem_degree}]
Suppose that $n=4$. Since the elements~$S_{ijk}$ are algebraically independent, we have $Y=\C^{10}$, hence, $y_0$ is a smooth point of~$Y$. Lemmas~\ref{lem_pre-images} and~\ref{lem_diff} imply that $y_0$ is a regular value of~$\varphi$, and has~$64$ pre-images. Hence the degree of~$\varphi$ is equal to~$64$. Therefore, $|K_1/K_2|=64$.

Now, suppose that $n>4$. Then $y_0$ is not a regular value of~$\varphi$, since it is not a smooth point of~$Y$. Lemmas~\ref{lem_pre-images} and~\ref{lem_diff} imply that near~$y_0$ the variety~$Y$ consists of smooth local components~$\varphi(U_{\omega,\sigma})$, where $U_{\omega,\sigma}$ is a small neighborhood of~$x_{\omega,\sigma}$ in the analytic topology. By Lemma~\ref{lem_diff2}, $\varphi(U_{\omega_1,\sigma_1})$ and~$\varphi(U_{\omega_2,\sigma_2})$ do not coincide near~$y_0$ unless $\omega_1=\omega_2$. On the other hand, for each~$\omega$, the local components~$\varphi(U_{\omega,\pm 1})$ coincide in a neighborhood of~$y_0$, since $S_{ijk}$ do not change under the simultaneous sign reversions of all~$s_{ij}$. Hence, near~$y_0$, the variety~$Y$ consists of $r_n$ smooth local components~$Y_{\omega}=\varphi(U_{\omega,\pm 1})$. A generic point~$y$ of any~$Y_{\omega}$ is a regular value of~$\varphi$. It has two pre-images belonging to~$U_{\omega,1}$ and~$U_{\omega,-1}$ respectively. Hence, the degree of~$\varphi$ is equal to~$2$. Therefore, $|K_1/K_2|=2$.
\end{proof}

\begin{proof}[Proofs of Theorems~\ref{theorem_degree_V} and~\ref{theorem_degree_edge} for $n\ge 5$]
$K_1$ is the field of rational functions in ${n+1\choose 2}$ independent variables~$s_{ij}$.
Let $K_1^{ev}\subset K_1$ be the subfield consisting of all rational functions for which both the numerator and the denominator contain only monomials of even total degrees. Then $K_1^{ev}$ is a proper subfield of~$K_1$. Obviously, $W\in K_1^{ev}$ if $n$ is even, and $W\notin K_1^{ev}$ but $W^2\in K_1^{ev}$ if $n$ is odd. Besides, $s_{ij}\notin K_1^{ev}$, but $s_{ij}s_{kl}\in K_1^{ev}$ for all $i,j,k,l$. Since all $S_{ijk}$ belong to~$K_1^{ev}$, we have $K_2\subset K_1^{ev}$. But by Theorem~\ref{theorem_degree}, $|K_1/K_2|=2$. Hence, $K_2=K^{ev}_1$, which completes the proofs of Theorems~\ref{theorem_degree_V} and~\ref{theorem_degree_edge} for $n\ge 5$.
\end{proof}

\begin{proof}[Proof of Theorem~\ref{theorem_degree_V} for $n=4$]
Consider the affine space~$\C^{11}$ with $10$ coordinates $S_{ijk}$ and one additional coordinate~$W$, and the regular morphism $\psi\colon X\to\C^{11}$ given by formulae~\eqref{eq_reg_map}, and by the Cayley--Menger formula for~$W$. Let $Z$ be the Zariski closure of~$\psi(X)$. The degree of~$K_1$ over the field~$K_2(W)$ obtained by adjoining~$W$ to~$K_2$ is equal to the degree of the mapping $\psi\colon X\to Z$.

It is easy to check that $W(x_{\emptyset,\pm 1})=2^{-10}\cdot 3^{-2}\cdot 5$, $W(x_{\omega,\pm 1})=-2^{-10}\cdot 3^{-1}$ if $|\omega|=1$, $W(x_{\omega,\pm 1})=-2^{-10}\cdot 3$ if $|\omega|=2$, and $W(x_{\gamma})=2^{-10}\cdot 5$ for all~$\gamma$. Then the point $z_0\in Z$ with all coordinates~$S_{ijk}$ equal to~$3/16$, and the coordinate~$W$  equal to~$2^{-10}\cdot 3^{-2}\cdot 5$ has $2$ pre-images~$x_{\emptyset,\pm 1}$ under~$\psi$.  Lemmas~\ref{lem_pre-images} and~\ref{lem_diff} imply that~$z_0$ is a regular value of~$\psi$. (Since all~$S_{ijk}$ and~$W$ do not change under the simultaneous sign reversions of all~$s_{ij}$, the images of neighborhoods~$U_{\emptyset,\pm 1}$ of~$x_{\emptyset,\pm 1}$ coincide near~$z_0$.)
Hence, $\psi$ has degree~$2$. Therefore, $|K_1/K_2(W)|=2$. But $|K_1/K_2|=64$. Thus, $|K_2(W)/K_2|=32$.
\end{proof}

\begin{proof}[Proof of Theorem~\ref{theorem_degree_edge} for $n=4$] 
Consider 
the affine space~$\C^{11}$ with $10$ coordinates $S_{ijk}$ and one  coordinate~$t$, and the regular morphism $\chi\colon X\to\C^{11}$ given by~\eqref{eq_reg_map}, and by~$t=s_{01}$. The degree of~$\chi$ considered as the mapping from~$X$ to the Zariski closure of~$\chi(X)$ is equal to the degree of~$K_1$ over~$K_2(s_{01})$. The point~$q_0\in\chi(X)$ with all~$S_{ijk}=3/16$, and $t=3$ has $4$ pre-images~$x_{\omega,1}$ such that $\{0,1\}\in\omega$. It can be easily checked that the differentials of~$\chi$ at these $4$ points are injective, and their images are pairwise distinct. Hence, a generic point $q\in\chi(X)$ close to~$q_0$ has one pre-image under~$\chi$, i.\,e., the degree of~$\chi$ is equal to~$1$. Therefore, $K_2(s_{01})=K_1$. Thus, $|K_2(s_{01})/K_2|=64$. Similarly, $|K_2(s_{ij})/K_2|=64$ for all $i\ne j$.
\end{proof}

\begin{proof}[Proof of Theorem~\ref{theorem_number}]
Since $K_2(s_{01})=K_1$, we obtain that $K_2(s_{01}^2)=K_1^{ev}$. Hence,  the minimal polynomial satisfied by~$s_{01}^2$ over~$K_2$ has degree~$32$. For generic values of  the  $2$-face areas~$A_{ijk}$,  this polynomial can have not more than $32$ roots. Since the edge lengths of any simplex $\Delta^n\subset\R^n$ are positive, this implies that the edge length~$\ell_{01}$ can take not more than $32$ values.
Besides, since $K_1^{ev}=K_2(s_{01}^2)$ we see that all~$s_{lm}^2$  are rational functions in the variables~$S_{ijk}$ and the variable~$s_{01}^2=\ell_{01}^4$. This means that for generic~$A_{ijk}$, the values of all~$\ell_{lm}$ can be uniquely restored from the values of all~$A_{ijk}$, and the value of~$\ell_{01}$. Thus, for generic~$A_{ijk}$, there are not more than $32$ possibilities for the set of the edge lengths~$\ell_{lm}$.
\end{proof}

\section{Negative results}\label{section_negative}

In this section we shall prove that $V$ is not integral over~$R_2$ for odd~$n\ge 5$, and that $V$ is not integral over~$\Lambda_2$ for $n=4$ and $n=5$. 

First, let $n=2q-1$ be odd, $n\ge 5$. Assume that $V$ were integral over~$R_2$. Then the element~$W=V^2$ would be integral over the $\Q$-algebra generated by the elements~$S_{ijk}=A_{ijk}^2$. Hence we would have a relation
\begin{equation}\label{eq_Wrel_int}
W^{M}+b_1(\bS_2)W^{M-1}+\cdots+b_M(\bS_2)=0,
\end{equation}
where $b_i(\bS_2)$ would be  polynomials in~$S_{ijk}$. The element~$W$ and all elements~$S_{ijk}$ are polynomials in independent variables~$s_{ij}$. We shall denote the set of these variables by~$\bs$, and shall write~$W(\bs)$ and~$S_{ijk}(\bs)$ to stress this dependance. The existence of a relation of the form~\eqref{eq_Wrel_int} would imply the following estimate:

\textit{For any\/  $C>0$ there is a\/ $D=D(C)>0$ such that for any set\/~$\bs$ of complex numbers, we have\/ $|W(\bs)|<D$ whenever\/ $|S_{ijk}(\bs)|<C$ for all\/~$i,j,k$.\/}

To obtain a contradiction we shall find a curve $\bs(t)$ such that all $S_{ijk}(\bs(t))$ are bounded, but $|W(\bs(t))|\to\infty$ as $t\to\infty$. We put:
\begin{align*}
&s_{ij}(t)=0&&\text{if either $0\le i,j<q$ or $q\le i,j<2q$,}\\
&s_{ij}(t)=t&&\text{if $0\le i<q\le j< 2q$ and $(i,j)\ne (r,q+r),$ $r=1,2,\ldots, q-1$,}\\
&s_{r,q+r}(t)=t+1&&\text{for $r=1,2,\ldots, q-1$.}
\end{align*}
It is easy to check that $S_{ijk}(\bs(t))=-1/16$ whenever the set $\{i,j,k\}$ contains one of the pairs $\{r,q+r\}$, $r=1,2,\ldots,q-1$, and $S_{ijk}(\bs(t))=0$ otherwise. On the other hand, the Cayley--Menger formula yields $W=(-4)^{1-q}(n!)^{-2}t$. This gives the required contradiction. Hence $V$ is not integral over~$R_2$.

Second, let $n$ be either $4$ or~$5$. If $V$ were integral over~$\Lambda_2$, then we would have a relation
\begin{equation*}%\label{eq_Wrel}
W^{M}+b_1(\bS_2)W^{M-1}+\cdots+b_M(\bS_2)=0,
\end{equation*}
where $b_i(\bS_2)$ would be Laurent polynomials in~$S_{ijk}$.  The existence of such relation would immediately imply the following estimate:

\textit{For any\/ $0<C_1<C_2$  there is a\/ $D=D(C_1,C_2)>0$ such that for any set\/~$\bs$ of complex numbers, we have\/ $|W(\bs)|<D$ whenever\/ $C_1<|S_{ijk}(\bs)|<C_2$ for all\/~$i,j,k$.}

To obtain a contradiction we shall find a curve $\bs(t)$ such that all $S_{ijk}^{\pm 1}(\bs(t))$ are bounded but $|W(\bs(t))|\to\infty$ as $t\to\infty$. Suppose, $n=5$. Let $a,b,c$ be non-zero complex numbers such that $a+b\ne 0$. We put:
\begin{gather*}
\ell_{01}(t)=\ell_{45}(t)=t^{-1},\qquad
\ell_{02}(t)=\ell_{12}(t)=\ell_{34}(t)=\ell_{35}(t)=bt,\\
\ell_{03}(t)=\ell_{13}(t)=\ell_{24}(t)=\ell_{25}(t)=at,\\
\ell_{04}(t)=\ell_{14}(t)=\ell_{05}(t)=\ell_{15}(t)=\ell_{23}(t)=(a+b)t-ct^{-3},
\end{gather*}
and put $s_{ij}(t)=\ell_{ij}^2(t)$ for all~$i,j$. As $t\to\infty$, we have
\begin{gather*}
S_{012}(\bs(t))=S_{345}(\bs(t))=\frac{1}{4}\,b^2 +O(t^{-4}),\\
S_{013}(\bs(t))=S_{245}(\bs(t))=\frac{1}{4}\,a^2 +O(t^{-4}),\\
S_{014}(\bs(t))=S_{015}(\bs(t))=S_{045}(\bs(t))=S_{145}(\bs(t))=\frac{1}{4}\,(a+b)^2 +O(t^{-4}),
\end{gather*}
\begin{gather*}
S_{023}(\bs(t))=S_{123}(\bs(t))=S_{234}(\bs(t))=S_{235}(\bs(t))=\frac{1}{2}\,abc(a+b)+O(t^{-4}),\\
W(\bs(t))=-\frac{1}{3600}\,a^2b^2(a+b)^2t^2+O(t^{-2}).
\end{gather*}
Hence, for large~$t$, all~$S_{ijk}^{\pm 1}(\bs(t))$ are bounded, and $|W(\bs(t))|\to\infty$ as $t\to\infty$, which yields a contradiction.

Now suppose that $n=4$. We take the same $s_{ij}(t)$, $0\le i,j\le 4$, as in the previous example, just forgetting about all~$s_{i5}$. Then all~$S_{ijk}^{\pm 1}(\bs(t))$ are bounded for large~$t$, but $W(\bs(t))=-\frac{1}{144}\,a^2b^2(a+b)^2t^4+O(1)$ as $t\to\infty$, which yields a contradiction.
Thus, for $n=4$ or~$5$, $V$ is not integral over~$\Lambda_2$.

\section{Valuations}\label{section_valuations}

An \textit{ordered commutative group\/} is a commutative group~$\Gamma$ with a total ordering~$<$ on it such that $\alpha<\beta$ implies $\alpha\gamma<\beta\gamma$ for all $\alpha,\beta,\gamma\in\Gamma$. (We use the multiplicative notation for~$\Gamma$ and denote the unity of~$\Gamma$ by~$1$.)

Let $F$ be a field, and let $F^{\times}$ be the multiplicative group of~$F$. A \textit{valuation\/} of~$F$ is a homomorphism $v$ of~$F^{\times}$ to an ordered commutative group $\Gamma$ such that 
%\begin{equation*}%\label{eq_val}
$v(x+y)\le \max(v(x),v(y))$
%\end{equation*}
for all $x,y\in F^{\times}$ such that $x+y\ne 0$. It is convenient to add to~$\Gamma$ an extra element~$0$ and to assume that $0<\gamma$ and $0\cdot\gamma=0$ for all $\gamma\in\Gamma$, and $0\cdot 0=0$. Then $v$ can be extended to the mapping $v\colon F\to\Gamma\cup\{0\}$ such that $v(0)=0$, $v(xy)=v(x)v(y)$, and  $v(x+y)\le \max(v(x),v(y))$ for all $x,y\in F$. Valuations of a field are in one-to-one correspondence with \textit{places\/} of this field, see~\cite[section~I.2]{Lan72}. The following lemma is a reformulation of Proposition~4,  p.~12 of~\cite{Lan72}.

\begin{lem}\label{lem_val}
Suppose that $F$ is a field, $R$ is a subring of~$F$, and $z$ is an element of~$F$. Then $z$ is integral over~$R$ if and only if $v(z)\le 1$ for all valuations~$v$ of~$F$ such that $v(x)\le 1$ for all~$x\in R$.
\end{lem}

To prove Theorem~\ref{theorem_R}, we shall apply Lemma~\ref{lem_val} to the field~$\F$, its subring~$R_2$, and the element~$V\in\F$. So we need to study valuations~$v$ of~$\F$ such that $v(x)\le 1$ for all $x\in R_2$. We denote by~$\fV$ the set of all valuations~$v$ satisfying this property. Since $\Q\subset R_2$, we see that for $v\in\fV$, 
 $v(\lambda)=1$ for all~$\lambda\in\Q^{\times}$. Hence $v(\lambda x)=v(x)$ for all~$\lambda\in\Q^{\times}$ and all $x\in\F$. By the Cayley--Menger formula, $V^2$ is a polynomial in~$\ell_{ij}$ with rational coefficients. Hence for any valuation~$v$ such that $v(\ell_{ij})\le 1$ for all~$i,j$, we have~$v(V)\le 1$. So we need to study valuations~$v\in\fV$ such that $v(\ell_{ij})>1$ for at least one pair~$\{i,j\}$. Denote the set of all such valuations by~$\fV_>$.

To a valuation~$v\in\fV_>$ we assign the following two elements of~$\Gamma=\Gamma(v)$: 
\begin{gather*}
\alpha(v)=\max_{0\le i<j\le n}v(\ell_{ij}),\\
\beta(v)=\max_{0\le i<j<k\le n}\min(v(\ell_{ij}),v(\ell_{jk}),v(\ell_{ki})).
\end{gather*}
Then $\alpha(v)>1$ and $\alpha(v)\ge\beta(v)$.
%When it is clear which valuation~$v$ is considered, we shall denote~$\alpha(v)$ and~$\beta(v)$ by $\alpha$ and~$\beta$ respectively.

{\sloppy
\begin{lem}\label{lem_triangle}
Let $v\in\fV_>$, and let $i,j,k\in\V$ be three pairwise distinct numbers. Suppose that $\gamma=v(\ell_{ij})\ge v(\ell_{jk})\ge v(\ell_{ki})=\delta$, and $\gamma>1$. Put
\begin{equation*}
\begin{aligned}
q_0&=\ell_{ij}+\ell_{jk}+\ell_{ki},&&&
q_1&=\ell_{ij}-\ell_{jk}-\ell_{ki},\\
q_2&=\ell_{ij}-\ell_{jk}+\ell_{ki},&&&
q_3&=\ell_{ij}+\ell_{jk}-\ell_{ki}.
\end{aligned}
\end{equation*}
Let $\nu_0,\nu_1,\nu_2,\nu_3$ be a permutation of the numbers $0,1,2,3$ such that $$v(q_{\nu_0})\ge v(q_{\nu_1})\ge v(q_{\nu_2})\ge v(q_{\nu_3}).$$ Then $v(\ell_{jk})=v(q_{\nu_0})=v(q_{\nu_1})=\gamma$. Besides, 
\begin{itemize}
\item if $\delta<\gamma^{-1}$, then $v(q_{\nu_3})\le v(q_{\nu_2})\le\gamma^{-1}$,
\item if $\delta\ge\gamma^{-1}$, then $v(q_{\nu_2})=\delta$ and $v(q_{\nu_3})\le\gamma^{-2}\delta^{-1}$.
\end{itemize}
\end{lem}
}

\begin{proof}
By Heron's formula, we have
$
A^2_{ijk}=-\frac{1}{16}\,q_0q_1q_2q_3
$.
Hence,
\begin{equation}\label{eq_4v}
v(q_0)v(q_1)v(q_2)v(q_3)=v(A_{ijk})^2\le 1,
\end{equation}
Since all $q_r$ are linear combinations of~$\ell_{ij}$, $\ell_{jk}$, and $\ell_{ki}$, we have $v(q_r)\le\gamma$.

If $v(\ell_{jk})<\gamma$, then $v(\ell_{ki})<\gamma$. Hence,  $v(q_r)=\gamma$,  $r=0,1,2,3$. Therefore, $\gamma^4=v(q_0)v(q_1)v(q_2)v(q_3)\le 1$, which is a contradiction. Thus, $v(\ell_{jk})=\gamma$.

We have, $v(q_0+q_1)=v(\ell_{ij})=\gamma$. Hence, either $v(q_0)=\gamma$ or $v(q_1)=\gamma$. Similarly, either $v(q_2)=\gamma$ or $v(q_3)=\gamma$. Therefore, $v(q_{\nu_0})=v(q_{\nu_1})=\gamma$. 

Suppose that $\delta=\gamma$. Since either $q_{\nu_2}-q_{\nu_3}$ or $q_{\nu_2}+q_{\nu_3}$ is equal to one of the elements~$\pm2\ell_{ij}$, $\pm2\ell_{jk}$, $\pm2\ell_{ki}$, we see that either $v(q_{\nu_2}-q_{\nu_3})$ or $v(q_{\nu_2}+q_{\nu_3})$ is equal to~$\gamma$. Hence, $v(q_{\nu_2})=\gamma$. Then, by~\eqref{eq_4v},  $v(q_{\nu_3})\le\gamma^{-3}$. 

Now, suppose that $\delta<\gamma$.
There exists $r\ne\nu_3$ such that either $q_r-q_{\nu_3}$ or $q_r+q_{\nu_3}$ is equal to~$\pm2\ell_{ki}$. Hence, either $v(q_r-q_{\nu_3})$ or $v(q_r+q_{\nu_3})$ is equal to~$\delta$. If $r$ were either~$\nu_0$ or~$\nu_1$, we would obtain that $v(q_{\nu_3})=\gamma$, which would contradict~\eqref{eq_4v}. Therefore, $r=\nu_2$. Consequently, 
either $v(q_{\nu_2})=\delta$ or $v(q_{\nu_2})=v(q_{\nu_3})>\delta$. In the first case, \eqref{eq_4v} implies that $v(q_{\nu_3})\le\gamma^{-2}\delta^{-1}$. In the second case, \eqref{eq_4v} implies that $v(q_{\nu_2})\le\gamma^{-1}$. In both cases, the lemma follows.
\end{proof}

\begin{lem}\label{lem_triangle2}
Let $v\in\mathfrak{V}_>$, let $\alpha=\alpha(v)$, and let $i,j,k\in\V$ be three pairwise distinct numbers.  Then  
$
v(\ell_{ij}\pm\ell_{jk}\pm\ell_{ki})\le\alpha^{-1}
$
for some choice of signs.
\end{lem}
\begin{proof}
If one of the three values~$v(\ell_{ij})$, $v(\ell_{jk})$, and~$v(\ell_{ki})$ is equal to~$\alpha$, then  the lemma follows  from Lemma~\ref{lem_triangle}. Assume that all three values~$v(\ell_{ij})$, $v(\ell_{jk})$, and~$v(\ell_{ki})$ are strictly less than~$\alpha$.

Assume that $v(\ell_{il})<\alpha$ for all $l\in I$, $l\ne i$. Take $l,m\in I$ such that $v(\ell_{lm})=\alpha$. Applying Lemma~\ref{lem_triangle} to the triple $(l,m,i)$, we obtain that either $v(\ell_{il})=\alpha$ or $v(\ell_{im})=\alpha$, which contradicts our assumption. Hence, there is an $l\in I$, $l\ne i$, such that $v(\ell_{il})=\alpha$. Applying Lemma~\ref{lem_triangle} to the triples $(l,i,j)$ and $(l,i,k)$, we obtain that $v(\ell_{jl})=v(\ell_{kl})=\alpha$, since both $v(\ell_{ij})$ and $v(\ell_{ki})$ are less than~$\alpha$. Now,  applying Lemma~\ref{lem_triangle} to the triples $(l,i,j)$, $(l,j,k)$, and $(l,k,i)$, we obtain that there exist  $\varepsilon_1, \varepsilon_1',\varepsilon_2, \varepsilon_2',\varepsilon_3, \varepsilon_3'\in\{-1,1\}$ such that the values of~$v$ on the elements
\begin{equation*}
q_1=\ell_{jk}+\varepsilon_1\ell_{jl}+\varepsilon_1'\ell_{kl},\quad
q_2=\ell_{ki}+\varepsilon_2\ell_{kl}+\varepsilon_2'\ell_{il},\quad
q_3=\ell_{ij}+\varepsilon_3\ell_{il}+\varepsilon_3'\ell_{jl}
\end{equation*}
do not exceed~$\alpha^{-1}$.
We put $q=-\varepsilon_1\varepsilon_3'q_1-\varepsilon_2'\varepsilon_3q_2+q_3$. Then $v(q)\le\alpha^{-1}$.

Assume that $\varepsilon_1\varepsilon_1'\varepsilon_2\varepsilon_2'\varepsilon_3\varepsilon_3'=1$.
Then 
$
q=-2\varepsilon_1\varepsilon_1'\varepsilon_3'\ell_{kl}+\ell_{ij}-\varepsilon_1\varepsilon_3'\ell_{jk}-\varepsilon_2'\varepsilon_3\ell_{ki}.
$
Since $v(\ell_{ij}-\varepsilon_1\varepsilon_3'\ell_{jk}-\varepsilon_2'\varepsilon_3\ell_{ki})<\alpha$ and $v(\ell_{kl})=\alpha$, we get $v(q)=\alpha>\alpha^{-1}$, which yields a contradiction.
Thus, $\varepsilon_1\varepsilon_1'\varepsilon_2\varepsilon_2'\varepsilon_3\varepsilon_3'=-1$.
Then
$
q=\ell_{ij}-\varepsilon_1\varepsilon_3'\ell_{jk}-\varepsilon_2'\varepsilon_3\ell_{ki}.
$
The lemma follows, since $v(q)\le \alpha^{-1}$.
\end{proof}

Let $v\in\fV_>$, and let $\alpha=\alpha(v)$. For $i,j\in I$, we shall use notation $i\sim_{\gamma}j$ to indicate that $v(\ell_{ij})\le\gamma$ and notation $i\approx_{\gamma}j$ to indicate that $v(\ell_{ij})<\gamma$.

\begin{cor}\label{cor_comp}
The relation~$\sim_{\gamma}$ \textnormal{(}respectively,~$\approx_{\gamma}$\textnormal{)} is an equivalence relation on the set~$I$ whenever $\gamma\ge\alpha^{-1}$ \textnormal{(}respectively, $\gamma>\alpha^{-1}$\textnormal{)}.
\end{cor}

\begin{proof}
We need only to prove that the relations~$\sim_{\gamma}$ and $\approx_{\gamma}$ are transitive.  Suppose that $i\sim_{\gamma}j$ and $j\sim_{\gamma}k$. Then $v(\ell_{ij})\le\gamma$ and $v(\ell_{jk})\le\gamma$. By Lemma~\ref{lem_triangle2},  there exist $\varepsilon,\varepsilon'\in\{-1,1\}$ such that 
$
v(\ell_{ij}+\varepsilon\ell_{jk}+\varepsilon'\ell_{ik})\le\alpha^{-1}$. Since $\gamma\ge\alpha^{-1}$, we obtain that $v(\ell_{ik})\le\gamma$, that is, $i\sim_{\gamma}k$. The case of the relation~$\approx_{\gamma}$ is similar.
\end{proof}

For a $\gamma\in\Gamma$, we shall denote by~$O(\gamma)$ (respectively, by~$o(\gamma)$) any element $x\in\F$ such that $v(x)\le\gamma$ (respectively, $v(x)<\gamma$). The equivalence classes of elements $i\in I$ with respect to~$\sim_{\gamma}$ (respectively, $\approx_{\gamma}$) will be called \textit{$O(\gamma)$-components\/} (respectively, \textit{$o(\gamma)$-components\/}). Then $i$ and $j$ belong to the same $O(\gamma)$-component (respectively, to the same $o(\gamma)$-component) if and only if $\ell_{ij}=O(\gamma)$ (respectively, $\ell_{ij}=o(\gamma)$). By Corollary~\ref{cor_comp}, the decomposition of~$I$ into $O(\gamma)$-components (respectively, $o(\gamma)$-components) is well defined whenever $\gamma\ge\alpha^{-1}$ (respectively, $\gamma>\alpha^{-1}$). Obviously, the decomposition into $o(\gamma)$-components is  finer than the decomposition into $O(\gamma)$-components, and the decomposition into $O(\gamma_1)$-components is  finer than the decomposition into $o(\gamma_2)$-components whenever $\gamma_1<\gamma_2$.

\section{A lemma from linear algebra}

\begin{lem}\label{lem_linear}
Let $I$ be a finite set.
Let $U$ be a vector space over~$\Q$, and let $u_{ij}$, $i,j\in I$, $i\ne j$, be vectors of~$U$ such that $u_{ij}=u_{ji}$. Assume that for each three pairwise distinct indices~$i, j,k\in I$, one of the four vectors~$u_{ij}\pm u_{jk}\pm u_{ki}$ vanishes. Then exactly one of the following two assertions holds:
\begin{enumerate}
\item There exist $w_i\in U$, $i\in I$, such that $u_{ij}=\pm(w_i-w_j)$ for all~$i$ and~$j$.
\item There exist a decomposition $I=I_0\sqcup I_1\sqcup I_2\sqcup I_3$, $I_r\ne\emptyset$,  and vectors $\xi,\eta,\zeta\in U\setminus\{0\}$, such that $\xi+\eta+\zeta=0$ and
$$
u_{ij}=\left\{
\begin{aligned}
&0&&\text{if}\,\ r(i)=r(j),\\
\pm\,&\eta&&\text{if}\,\ \{r(i),r(j)\}=\{0,1\}\,\ \text{or}\,\ \{2,3\},\\
\pm\,&\xi&&\text{if}\,\ \{r(i),r(j)\}=\{0,2\}\,\ \text{or}\,\ \{1,3\},\\
\pm\,&\zeta&&\text{if}\,\ \{r(i),r(j)\}=\{0,3\}\,\ \text{or}\,\ \{1,2\},
\end{aligned}
\right.
$$
where $r(i)$ denotes the number such that $i\in I_{r(i)}$.
\end{enumerate}
\end{lem} 
 
\begin{proof}
If $u_{ij}=0$, then $u_{ik}=\pm u_{jk}$ for all $k\ne i,j$. Hence the claim of the lemma for the set~$I$ follows from the claim of the lemma for the set~$I\setminus\{j\}$. Therefore we may assume  that $u_{ij}\ne 0$ for all~$i\ne j$. Then~(ii) may hold only if~$|I|=4$.

Obviously, (i) holds whenever $|I|\le 3$.

Suppose that $|I|=4$. For convenience, we assume that $I=\{0,1,2,3\}$.
We have 
\begin{equation}
\begin{aligned}
\label{eq_u}u_{01}+\varepsilon_{02}u_{02}+\varepsilon_{12}u_{12}&=0,&&&
u_{01}+\varepsilon_{03}u_{03}+\varepsilon_{13}u_{13}&=0,\\
u_{23}+\delta_{02}u_{02}+\delta_{03}u_{03}&=0,&&&
u_{23}+\delta_{12}u_{12}+\delta_{13}u_{13}&=0
\end{aligned}
\end{equation}
for some $\varepsilon_{ij},\delta_{ij}\in\{-1,1\}$.

Suppose, $\varepsilon_{02}\delta_{02}=-\varepsilon_{03}\delta_{03}$. Then (i) holds for $w_0=0$, $w_1=u_{01}$, $w_2=-\varepsilon_{02}u_{02}$, and $w_3=-\varepsilon_{03}u_{03}$. Similarly, (i) holds if at least one of  $\varepsilon_{12}\delta_{12}=-\varepsilon_{13}\delta_{13}$, $\varepsilon_{02}\delta_{02}=-\varepsilon_{12}\delta_{12}$, and $\varepsilon_{03}\delta_{03}=-\varepsilon_{13}\delta_{13}$ holds.

Suppose that~(i) does not hold. Then $\varepsilon_{02}\delta_{02}=\varepsilon_{03}\delta_{03}=\varepsilon_{12}\delta_{12}=\varepsilon_{13}\delta_{13}.$ We denote this number by~$\sigma$. Taking the linear combination of equations~\eqref{eq_u} with the coefficients~$1$, $1$, $-\sigma$, and~$-\sigma$ respectively, we obtain that $u_{01}=\sigma u_{23}$. Similarly, we obtain that $u_{02}=\pm u_{13}$ and $u_{03}=\pm u_{12}$. Therefore, (ii) holds for $I_r=\{r\}$, $\xi=\varepsilon_{02}u_{02}$, $\eta=u_{01}$, and $\zeta=\varepsilon_{12}u_{12}$.

Now suppose that $|I|\ge 5$. The lemma has already been proved for all $4$-element subsets~$J\subset I$. We shall say that a $4$-element subset~$J$ is of the \textit{first type\/} if~(i) holds for it, and is of the \textit{second type\/} if~(ii) holds for it.

\begin{lem}\label{lem_first_type_subsets1}
Suppose that there exist\/ $i,j\in I$, $i\ne j$, such that all\/ $4$-element subsets $J\subset I$ containing both~$i$ and~$j$ are of the first type. Then~\textnormal{(i)} holds for~$I$.
\end{lem}
\begin{proof}
Consider any subset $J=\{i,j,k,l\}$. Since $J$ is of the first type, there exist vectors~$w_{i}^{J}$, $w_{j}^{J}$, $w_{k}^{J}$, and~$w_{l}^{J}$ such that
\begin{equation}\label{eq_uw}
u_{mn}=\pm(w_m^{J}-w_n^{J}),\quad m,n\in\{i,j,k,l\},\ m\ne n.
\end{equation}
Subtracting $w_i^J$ from all vectors~$w_{i}^{J}$, $w_{j}^{J}$, $w_{k}^{J}$, and~$w_{l}^{J}$ and multiplying (if necessary) all these vectors by~$-1$, we may achieve that $w_i^J=0$ and $w_j^J=u_{ij}$, while equations~\eqref{eq_uw} still hold. 

Now assume that $w_k^{J_1}\ne w_k^{J_2}$ for two different $4$-element subsets $J_1=\{i,j,k,l_1\}$ and $J_2=\{i,j,k,l_2\}$. By~\eqref{eq_uw} for $m=i$ and~$n=k$, we obtain that both~$w_k^{J_1}$ and~$w_k^{J_2}$ are equal to~$\pm u_{ik}$. Hence $w_k^{J_1}=-w_k^{J_2}$. Similarly, by~\eqref{eq_uw} for $m=j$ and $n=k$, we obtain that
$
(w_k^{J_1}-u_{ij})=-(w_k^{J_2}-u_{ij}).
$ 
Therefore $u_{ij}=0$, which contradicts our assumption.
Thus, the elements~$w^J_k$ coincide for all $4$-element subsets~$J$ containing~$i$, $j$, and~$k$. Denoting these elements by~$w_k$, we obtain that the equations $u_{kl}=\pm(w_k-w_l)$ hold for all~$k$ and~$l$.  
\end{proof}

\begin{lem}\label{lem_first_type_subsets2}
If\/ $|I|\ge 5$, then all $4$-element subsets $J\subset I$ are of the first type.
\end{lem}
\begin{proof}
Let $K=\{i,j,k,l,m\}$ be an arbitrary $5$-element subset of~$I$. 
Suppose that $K$ contains three $4$-element subsets of the second type, say $\{i,j,k,l\}$, $\{i,j,k,m\}$, and $\{i,j,l,m\}$.  Then 
$
u_{ij}=\pm u_{kl}=\pm u_{km}=\pm u_{lm}
$
for certain choice of signs. But $u_{kl}\pm u_{km}\pm u_{lm}=0$. Hence, $u_{kl}=0$, which contradicts our assumption.

Therefore, $K$ contains at least three $4$-element subsets of the first type. Then Lemma~\ref{lem_first_type_subsets1} implies that assertion~(i) holds for~$K$. Therefore all $4$-element subsets of~$K$ are of the first type. 
\end{proof}

It follows from Lemmas~\ref{lem_first_type_subsets1} and~\ref{lem_first_type_subsets2} that (i) holds for~$I$ if $|I|\ge 5$ and  all $u_{ij}$ are non-zero. This completes the proof of Lemma~\ref{lem_linear}.
\end{proof}

\section{Two types of valuations}\label{section_types}

\begin{lem}\label{lem_types}
Let $v\in\fV_>$, and let $\alpha=\alpha(v)$, $\beta=\beta(v)$. Then $v$ satisfies exactly one of the following two conditions:
\begin{enumerate}
\item There exist numbers $\varepsilon_{ij}=-\varepsilon_{ji}\in\{-1,1\}$, $i,j\in\V$, $i\ne j$, such that 
$$
v(\varepsilon_{ij}\ell_{ij}+\varepsilon_{jk}\ell_{jk}+\varepsilon_{ki}\ell_{ki})\le\alpha^{-1}
$$ 
for all pairwise distinct $i,j,k\in\V$.

\item There exist a decomposition
$\V=\V_0\sqcup\V_1\sqcup\V_2\sqcup\V_3$, $\V_r\ne\emptyset$, and elements $a,b\in\F$ such that
\begin{itemize}
\item $v(a)=v(a+b)=\alpha$, $v(b)=\beta>\alpha^{-1}$,

\item $v(\ell_{ij})\le\alpha^{-1}$ whenever $i$ and $j$ belong to the same set~$\V_r$,

\item There exist numbers $\varepsilon_{ij}=\varepsilon_{ji}\in\{-1,1\}$ such that 
\begin{align}
&\label{eq_est_1}v(\varepsilon_{ij}\ell_{ij}-b)\le\alpha^{-2}\beta^{-1}&&\text{if}\ \ \{r(i),r(j)\}=\{0,1\}\text{ or\/ }\{2,3\},\\ 
&\label{eq_est_2}v(\varepsilon_{ij}\ell_{ij}-a)\le\alpha^{-2}\beta^{-1}&&\text{if}\ \ \{r(i),r(j)\}=\{0,2\}\text{ or\/ }\{1,3\},\\
&\label{eq_est_3}v(\varepsilon_{ij}\ell_{ij}-a-b)\le\alpha^{-2}\beta^{-1}&&\text{if}\ \ \{r(i),r(j)\}=\{0,3\}\text{ or\/ }\{1,2\},
\end{align}
where $r(i)$ denotes the number such that $i\in I_{r(i)}$.
\end{itemize}
\end{enumerate}
\end{lem}
\begin{proof}
Let $X$ be the set of all~$x\in\F$ such that $v(x)\le\alpha^{-1}$. Then $X$ is a vector space over~$\Q$. Let $U=\F/X$, and let $u_{ij}$ be the images of~$\ell_{ij}$ under the projection $\F\to U$. By Lemma~\ref{lem_triangle2}, for any pairwise distinct $i,j,k\in I$, one of the four vectors $u_{ij}\pm u_{jk}\pm u_{ki}$ vanishes.  Then either assertion~(i) or assertion~(ii) of Lemma~\ref{lem_linear} holds for the vectors~$u_{ij}$. 

1. Suppose, assertion~(i) of Lemma~\ref{lem_linear} holds for~$u_{ij}$. Then $u_{ij}=\varepsilon_{ij}(w_i-w_j)$ for some vectors $w_i\in U$ and some numbers~$\varepsilon_{ij}\in\{-1,1\}$ such that $\varepsilon_{ji}=-\varepsilon_{ij}$. Therefore, for any pairwise distinct $i$, $j$, and~$k$, we have $\varepsilon_{ij}u_{ij}+\varepsilon_{jk}u_{jk}+\varepsilon_{ki}u_{ki}=0$, hence, $\varepsilon_{ij}\ell_{ij}+\varepsilon_{jk}\ell_{jk}+\varepsilon_{ki}\ell_{ki}=O(\alpha^{-1})$.

2. Suppose, assertion~(ii) of Lemma~\ref{lem_linear} holds for~$u_{ij}$. Choose arbitrarily elements $k_r\in I_r$, $r=0,1,2$, and put $a=\pm\ell_{k_0k_2}$, $b=\pm\ell_{k_0k_1}$, $c=\pm\ell_{k_1k_2}$, where the signs are chosen so that $a$, $b$, and~$c$ go to $\xi$, $\eta$, and~$\zeta$ respectively under the projection $\F\to U$. Since $\xi,\eta,\zeta\ne 0$, we have $v(a),v(b),v(c)>\alpha^{-1}$. Renumbering the subsets~$I_0$, $I_1$, and~$I_2$, we may achieve that $v(a)\ge v(c)\ge v(b)$. We have $v(\ell_{ij})\le\alpha^{-1}$ whenever $r(i)=r(j)$, $v(\ell_{ij})=v(b)$ whenever $\{r(i),r(j)\}$ is either $\{0,1\}$ or $\{2,3\}$, 
$v(\ell_{ij})=v(a)$ whenever $\{r(i),r(j)\}$ is either $\{0,2\}$ or $\{1,3\}$, and
$v(\ell_{ij})=v(c)$ whenever $\{r(i),r(j)\}$ is either $\{0,3\}$ or $\{1,2\}$. It follows easily that $v(a)=\alpha$ and $v(b)=\beta$. Hence, $\beta>\alpha^{-1}$. Applying Lemma~\ref{lem_triangle} to the triple~$k_0,k_1,k_2$, we obtain that $v(c)=v(a)=\alpha$ and $v(a+b+c)\le\alpha^{-2}\beta^{-1}$. Hence, $v(a+b)=\alpha$. 

Let us prove~\eqref{eq_est_1}--\eqref{eq_est_3}.   Let $\widetilde{X}$ be the $\Q$-vector subspace of~$\F$ consisting of all~$x$ such that $v(x)\le\alpha^{-2}\beta^{-1}$, let $\widetilde{U}=\F/\widetilde{X}$, and let $\widetilde{u}_{ij}$ be the image of~$\ell_{ij}$ under the projection~$\F\to\widetilde{U}$. 
Let $i,j,k$ be arbitrary elements of~$I$ belonging to three pairwise distinct subsets~$I_r$. Then Lemma~\ref{lem_triangle} yields that $\ell_{ij}\pm\ell_{jk}\pm\ell_{ki}=O(\alpha^{-2}\beta^{-1})$ for some choice of signs. Hence $\widetilde{u}_{ij}\pm\widetilde{u}_{jk}\pm\widetilde{u}_{ki}=0$. Now, consider arbitrary four elements $i_r\in I_r$, $r=0,1,2,3$. Then the $6$ vectors $\widetilde{u}_{i_ri_t}$ satisfy the condition of Lemma~\ref{lem_linear}. (Notice that it is not true that all vectors $\widetilde{u}_{ij}$, $i,j\in I$, $i\ne j$, satisfy the condition of Lemma~\ref{lem_linear}.) If the $4$-element subset $\{i_0,i_1,i_2,i_3\}$ was of the first type with respect to the vectors~$\widetilde{u}_{i_ri_t}$, it would also be of the first type with respect to the vectors~$u_{i_ri_t}$, which are the images of~$\widetilde{u}_{i_ri_t}$ under the projection~$\widetilde{U}\to U$. But the subset $\{i_0,i_1,i_2,i_3\}$ is of the second type with respect to the vectors~$u_{i_ri_t}$. Hence, it is also of the second type with respect to the vectors~$\widetilde{u}_{i_ri_t}$. Therefore, $\widetilde{u}_{i_2i_3}=\pm\widetilde{u}_{i_0i_1}$, $\widetilde{u}_{i_1i_3}=\pm\widetilde{u}_{i_0i_2}$, and $\widetilde{u}_{i_1i_2}=\pm\widetilde{u}_{i_0i_3}$.
Thus $\ell_{i_2i_3}=\pm\ell_{i_0i_1}+O(\alpha^{-2}\beta^{-1})$, $\ell_{i_1i_3}=\pm\ell_{i_0i_2}+O(\alpha^{-2}\beta^{-1})$, and $\ell_{i_1i_2}=\pm\ell_{i_0i_3}+O(\alpha^{-2}\beta^{-1})$. Taking $i_0=k_0$ and $i_1=k_1$, we obtain that $\ell_{i_2i_3}=\pm b+O(\alpha^{-2}\beta^{-1})$. Hence, for arbitrary $i_0\in I_0$, $i_1\in I_1$, we get $\ell_{i_0i_1}=\pm b+O(\alpha^{-2}\beta^{-1})$. The proofs of~\eqref{eq_est_2}, \eqref{eq_est_3} are similar.
\end{proof}

Valuations satisfying condition~(i) will be called \textit{valuations of type~I\/}, and valuations satisfying condition~(ii)  will be called \textit{valuations of type~II\/}.

\section{Valuations of type I}

\begin{lem}\label{lem_typeI}
Let $v\in\fV_>$ be a valuation of type~I, let $\alpha=\alpha(v)$, $\beta=\beta(v)$. Then
\begin{enumerate}
\item $v(V)\le\alpha$,
\item $v(V)\le 1$ if $n$ is even,
\item $v(V)\le\max\left(1,\alpha^{-\frac{n-1}{4}}\beta^{-\frac{n+3}{4}}\right)$ if $n\equiv 1\pmod 4$,
\item $v(V)\le\max\left(1,\alpha^{-\frac{n-3}{4}}\beta^{-\frac{n+1}{4}}\right)$ if $n\equiv 3\pmod 4$.
\item $v(P_{i_1i_2i_3i_4i_5}V^2)\le 1$ for any pairwise distinct indices $i_1,\ldots,i_5\in I$, where $P_{i_1i_2i_3i_4i_5}$ is the polynomial defined by~\eqref{eq_P}.
\end{enumerate}
\end{lem}

\begin{proof}
For any~$i\ne j$, we put $a_{ij}=\varepsilon_{ij}\ell_{ij}$. Then $a_{ij}=-a_{ji}$, and for any pairwise distinct~$i,j,k$, we have $v(a_{ij}+a_{jk}+a_{ki})\le\alpha^{-1}$.

The set $I$ consists of at least two $o(\alpha)$-components. Let $J$ be the $o(\alpha)$-component of the smallest cardinality, which we conveniently define by~$k+1$. Then $k\le (n-1)/2$. Renumbering the elements of~$I$ we may achieve that $J=\{0,1,\ldots,k\}$. 

The Cayley--Menger formula~\eqref{eq_CM} easily implies that $V^2=\frac{(-1)^n}{2^n(n!)^2}\det C$, where $C=(c_{ij})_{i,j=1}^n$ is the matrix with entries 
\begin{equation*}
c_{ii}=-2a_{0i}^2,\qquad
c_{ij}=a_{ij}^2-a_{0i}^2-a_{0j}^2,\ i\ne j.
\end{equation*}
Then $v(V)^2=v(\det C)$.
We have $C=X+Y$, where $x_{ij}=-2a_{0i}a_{0j}$ for all $i$ and~$j$,
$y_{ii}=0$ for all~$i$, and  $y_{ij}=a_{ij}^2-(a_{0i}-a_{0j})^2$ for $i\ne j$. 
Since $X$ is a rank~$1$ matrix,
\begin{equation}\label{eq_detB}
\det C=\det Y+\sum_{i,j=1}^n(-1)^{i+j}x_{ij}\det Y_{ij}=\det Y-2\sum_{i,j=1}^n(-1)^{i+j}a_{0i}a_{0j}\det Y_{ij},
\end{equation}
where $Y_{ij}$ is the matrix~$Y$ with the $i$th row and the $j$th column deleted.

Now, let us estimate the values~$v(y_{ij})$. First, for any $i$ and $j$, we have 
$$
v(y_{ij})=v(a_{ij}+a_{0i}-a_{0j})v(a_{ij}-a_{0i}+a_{0j})\le \frac{1}\alpha\cdot\alpha=1.
$$
This implies that $v(\det Y)\le 1$ and $v(\det Y_{ij})\le 1$. Since $v(a_{0i})\le\alpha$,  formula~\eqref{eq_detB} yields $v(\det C)\le \alpha^2$. Hence, $v(V)\le\alpha$.

To obtain a better estimate for~$v(V)$, we need better estimates for~$v(y_{ij})$. Suppose that both $i$ and~$j$ are greater than~$k$. Then $v(a_{0i})=v(a_{0j})=\alpha$. Since $v(a_{ij}+a_{0i}-a_{0j})\le \alpha^{-1}$, we obtain that $v(a_{ij}-a_{0i}-a_{0j})=v(a_{ij}+a_{0i}+a_{0j})=\alpha$. Now, by Heron's formula, we have 
$v(A_{0ij})^2=\alpha^2 v(y_{ij})$. But $v(A_{0ij})\le 1$.
Hence, $v(y_{ij})\le \alpha^{-2}$. Thus, the matrix~$Y$ has the form shown in Figure~\ref{fig_Y}(a). We see that $Y$ has a submatrix of size $(n-k)\times (n-k)$ consisting of elements that are $O(\alpha^{-2})$, while all other elements are~$O(1)$. Hence, every~$Y_{ij}$ has a submatrix of size at least $(n-k-1)\times (n-k-1)$ consisting of elements that are $O(\alpha^{-2})$. 

\begin{figure}
\begin{center}

\unitlength=9.5mm
\begin{picture}(11.5,5.8)

\put(1.4,.5){(a)}

\put(7.6,0){(b)}

\put(0,2.3){%
\begin{picture}(0,0)
\put(0,0){\line(0,1){3}}
\put(3,0){\line(0,1){3}}
\put(1.2,0){\line(0,1){3}}

\put(0,0){\line(1,0){3}}
\put(0,3){\line(1,0){3}}
\put(0,1.8){\line(1,0){3}}

\put(.2,2.28){$O(1)$}
\put(.2,.8){$O(1)$}
\put(1.7,2.28){$O(1)$}
\put(1.6,.8){$O\bigl(\frac{1}{\alpha^2}\bigr)$}

\put(3.2,.81){$\left.\rule{0pt}{9.7mm}\right\}n-k$}
\put(3.2,2.31){$\left.\rule{0pt}{6mm}\right\}k$}

\put(0.05,-.3){$\underbrace{\rule{10.5mm}{0pt}}_{\textstyle k}$}
\put(1.25,-.3){$\underbrace{\rule{16.2mm}{0pt}}_{\textstyle n-k}$}
\end{picture}
}

\put(5.5,.8){%
\begin{picture}(4,5)
\put(0,1){%
\begin{picture}(4,4)

\put(0,0){\line(1,0){4}}
\put(0,0){\line(0,1){4}}
\put(4,4){\line(-1,0){4}}
\put(4,4){\line(0,-1){4}}

\put(0,2.1){\line(1,0){4}}
\put(1.9,0){\line(0,1){4}}

\put(1.9,3.2){\line(1,0){2.1}}
\put(.8,0){\line(0,1){2.1}}

\put(0.6,2.95){$O(1)$}
\put(2.6,3.5){$O(1)$}

\put(0.03,1){$O(1)$}
\put(.83,1){$O\bigl(\frac{1}{\alpha\beta}\bigr)$}

\put(2.4,2.55){$O\bigl(\frac{1}{\alpha\beta}\bigr)$}
\put(2.45,1){$O\bigl(\frac{1}{\alpha^2}\bigr)$}

\put(4.2,3.5){$\left.\rule{0pt}{4.2mm}\right\}\,m$}
\put(4.2,2.55){$\left.\rule{0pt}{5.9mm}\right\}\frac{n-1}{2}-m$}
\put(4.2,.95){$\left.\rule{0pt}{11mm}\right\}\frac{n+1}{2}$}

\put(0.07,-.3){$\underbrace{\rule{4.5mm}{0pt}}_{\displaystyle m}$}
\put(0.8,-.3){$\underbrace{\rule{9mm}{0pt}}_{\frac{n-1}{2}-m}$}
\put(1.97,-.3){$\underbrace{\rule{19mm}{0pt}}_{\textstyle\frac{n+1}{2}}$}

\end{picture}}

\end{picture}}

\end{picture}
\end{center}
\caption{The matrix~$Y$}\label{fig_Y}
\end{figure}
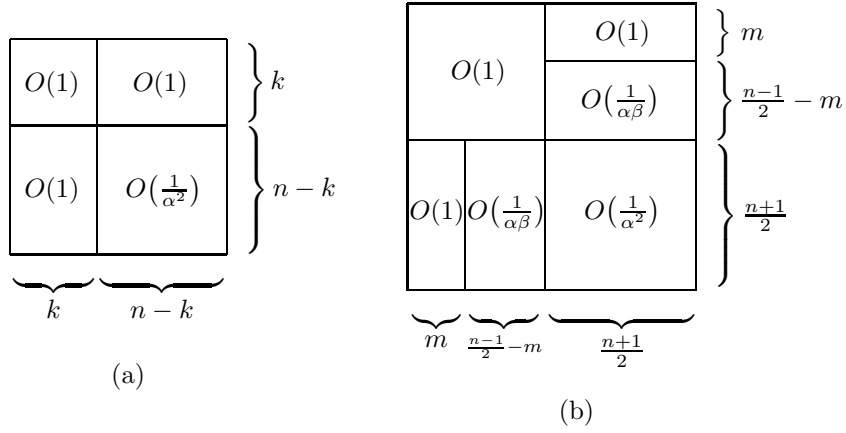

If $k<(n-1)/2$, then $n-k>n/2$ and $n-k-1>(n-1)/2$. We obtain that $\det Y=O(\alpha^{-2})$ and $\det Y_{ij}=O(\alpha^{-2})$ for all $i$ and $j$. Since $v(a_{0i})\le \alpha$, formula~\eqref{eq_detB} implies that $v(\det C)\le 1$. Therefore, $v(V)\le 1$.

The only case which we still need to handle is $k=(n-1)/2$. Obviously, this may happen only if $n$ is odd, and $I$ consists of exactly two $o(\alpha)$-components, either consisting of exactly $(n+1)/2$ elements. We may assume that these components are $J_1=\{0,1,\ldots,(n-1)/2\}$ and $J_2=\{(n+1)/2,(n+3)/2,\ldots,n\}$. Then $\beta<\alpha$. 

If $\beta\le\alpha^{-1}$, then~(i) implies~(iii) and~(iv). Assume that $\beta>\alpha^{-1}$. Then the decomposition of~$I$ into $o(\beta)$-components is well defined. By definition of $\beta$, at least one of the two $o(\alpha)$-components, say~$J_1$, consists of at least two $o(\beta)$-components. Let $K$ be the $o(\beta)$-component contained in~$J_1$ with the smallest cardinality. Then the cardinality of $K$ does not exceed~$(n+1)/4$. We may assume that $K=\{0,1,\ldots,m\}$, $m\le (n-3)/4$.

Suppose that $m<i\le k<j$. Then $v(a_{0i})=\beta$ and $v(a_{0j})=\alpha$.
Since $v(a_{ij}+a_{0i}-a_{0j})\le \alpha^{-1}$, we obtain that $v(a_{ij}-a_{0i}-a_{0j})=\beta$ and $v(a_{ij}+a_{0i}+a_{0j})=\alpha$. Since $v(A_{0ij})\le 1$, we obtain that $v(y_{ij})\le \alpha^{-1}\beta^{-1}$. Thus, the matrix~$Y$ has the form shown in Figure~\ref{fig_Y}(b).

First, we get $\det Y=O(\alpha^{-2})$. Second,  $\det Y_{ij}=O(\alpha^{-2})$ unless both $i$ and~$j$ are greater than $(n-1)/2$. In the latter case the $(n-1)\times (n-1)$ matrix~$Y_{ij}$ has the lower-right $(n-1)/2\times (n-1)/2$ submatrix consisting of elements that are $O(\alpha^{-2})$. Hence,
$$
\det Y_{ij}=(-1)^{\frac{n-1}{2}}\det Y_{ij}'\det Y_{ij}''+O(\alpha^{-2}),
$$
where $Y_{ij}'$ and $Y_{ij}''$ are the lower-left and the upper-right $(n-1)/2\times (n-1)/2$ submatrices of~$Y_{ij}$. But $Y_{ij}'$ has $\frac{n-1}{2}-m$ columns consisting of elements that are $O(\alpha^{-1}\beta^{-1})$, and $Y_{ij}''$ has $\frac{n-1}{2}-m$ rows consisting of elements that are $O(\alpha^{-1}\beta^{-1})$. Hence, $v(\det Y_{ij}')$ and $v(\det Y_{ij}'')$ do not exceed $(\alpha\beta)^{-\frac{n-1}{2}+m}$.
Therefore,
$$
\det Y_{ij}\le\max\left(\alpha^{-2},(\alpha\beta)^{-n+1+2m}\right)
$$
Using~\eqref{eq_detB}, we finally obtain
$$
v(V)^2=v(\det C)\le \max\left(1,\alpha^{-n+3+2m}\beta^{-n+1+2m}\right).
$$
Assertions~(iii) and~(iv) follow, since $m\le [(n-3)/4]$.

Now, let us prove~(v). Again, we need to consider only the case of two $o(\alpha)$-com\-po\-nents~$J_1$ and~$J_2$ of equal cardinalities~$(n+1)/2$, since in all other cases we have $v(V)\le 1$, which trivially implies~(v). So $n$ is odd and $n\ge 5$.

Suppose that $\beta^2\ge\alpha^{-1}$. Then  $\alpha^{-(n-1)/4}\beta^{-(n+3)/4}\le 1$ if $n\equiv 1\pmod 4$ and $\alpha^{-(n-3)/4}\beta^{-(n+1)/4}\le 1$ if $n\equiv 3\pmod 4$. Hence~(iii) and~(iv) imply that $v(V)\le 1$, which again implies~(v).

Suppose that $\beta^2<\alpha^{-1}$.  Three of the five numbers  $i_1,\ldots,i_5$, say $i_1$, $i_2$, and~$i_3$, belong to the same $o(\alpha)$-component. Then $v(a_{i_1i_2})$, $v(a_{i_2i_3})$, and $v(a_{i_1i_3})$ do not exceed~$\beta$. Then $
v(S_{i_1i_2i_3})\le\beta^4<\alpha^{-2}.
$
Since $v(S_{jkm})\le 1$ for all $j,k,m$, this implies that $v(P_{i_1i_2i_3i_4i_5})<\alpha^{-2}$.
But $v(V)\le\alpha$. Therefore, $v(P_{i_1i_2i_3i_4i_5}V^2)<1$.
\end{proof}

\section{Valuations of type II}\label{section_typeII}

\begin{lem}\label{lem_typeII}
Let $v\in\fV_>$ be a valuation of type~II, let $\alpha=\alpha(v)$, $\beta=\beta(v)$. Then
\begin{enumerate}
\item $v(V)\le(\alpha\beta)^{\frac{6-n}{2}}$ if $n$ is even, in particular, $v(V)\le 1$ if $n$ is even and\/ $\ge 6$,
\item $v(V)\le\alpha^{\frac{7-n}{2}}\beta^{\frac{5-n}2}$ if $n$ is odd,
\item $v(P_{i_1i_2i_3i_4i_5}V^2)\le 1$ for any pairwise distinct~$i_1,\ldots,i_5\in I$ if $n\ge 6$,
\item $v(P_{i_1i_2i_3i_4i_5}DV^2)\le 1$ for any pairwise distinct~$i_1,\ldots,i_5\in I$ if $n=5$,
\item $v(QV^2)\le 1$ if $n=4$,
\end{enumerate}
where $P_{i_1i_2i_3i_4i_5}$, $D$, $Q$ are the polynomials defined by~\eqref{eq_P},~\eqref{eq_D},~\eqref{eq_Q} respectively.
\end{lem}

\begin{proof}
Renumbering elements of~$I$ we may achieve that 
\begin{align*}
I_0&=\{0,1,\ldots,k_1-1\},&
I_1&=\{k_1,k_1+1,\ldots,k_2-1\},\\
I_2&=\{k_2,k_2+1,\ldots,k_3-1\},&
I_3&=\{k_3,k_3+1,\ldots,n\}.
\end{align*}
We shall conveniently put $k_0=0$ and $k_4=n+1$.
From~\eqref{eq_est_1}--\eqref{eq_est_3}, we obtain that  
\begin{equation}\label{eq_IIsq}
\begin{aligned}
\ell_{ij}^2&=b^2+O(\alpha^{-2})&
&\text{if $\{r(i),r(j)\}=\{0,1\},\{2,3\}$},\\
\ell_{ij}^2&=a^2+O(\alpha^{-1}\beta^{-1})&
&\text{if $\{r(i),r(j)\}=\{0,2\},\{1,3\}$},\\
\ell_{ij}^2&=(a+b)^2+O(\alpha^{-1}\beta^{-1})&
&\text{if $\{r(i),r(j)\}=\{0,3\},\{1,2\}$}.
\end{aligned}
\end{equation}

As in the proof of Lemma~\ref{lem_typeI}, we see that $v(V)^2=v(\det C)$, where $C=(c_{ij})_{i,j=1}^n$ is the matrix with entries 
\begin{equation*}
c_{ii}=-2\ell_{0i}^2,\qquad
c_{ij}=\ell_{ij}^2-\ell_{0i}^2-\ell_{0j}^2,\ i\ne j.
\end{equation*}
Apply to~$C$ the following elementary transformations that do not change the determinant of it: For $r=1,2,3$ and for $i=k_r+1,\ldots,k_{r+1}-1$, subtract the $k_r$th row from the $i$th row, and subtract the $k_r$th column from the $i$th column. The result is independent of the order in which these transformations are performed. Finally, subtract the $k_2$th row from the $k_3$th row, and subtract the $k_2$th column from the $k_3$th column. Let $C'=(c_{ij}')$ be the matrix obtained. Let us estimate  the entries of~$C'$. Since $C'$ is symmetric, we shall consider only~$c_{ij}'$ for $1\le i\le j\le n$.

1. Suppose, $i=j$. If $i\ne k_1,k_2,k_3$, we have $c_{ii}'=-2\ell_{k_{r(i)}i}^2=O(\alpha^{-2})$.
Besides, 
\begin{equation*}
c_{k_1k_1}'=-2\ell_{0k_1}^2=O(\beta^2),\quad
c_{k_2k_2}'=-2\ell_{0k_2}^2=O(\alpha^2),\quad
c_{k_3k_3}'=-2\ell_{k_2k_3}^2=O(\beta^2).
\end{equation*}

2. Suppose that $i<j$ and $i$ and $j$ belong to the same subset $I_r$. If $i\ne k_r$, then 
$$
c_{ij}'=\ell_{ij}^2-\ell_{k_ri}^2-\ell_{k_rj}^2=O(\alpha^{-2}).
$$
If $i=k_r$, then
$
c_{k_rj}'=\ell_{k_rj}^2-\ell_{mj}^2+\ell_{mk_r}^2$,
where $m=0$ if  $r=1$ or~$2$, and $m=k_2$ if $r=3$.
By~\eqref{eq_IIsq}, $c_{k_rj}'=O(\alpha^{-2})$ if $r=1$ or~$3$, and $c_{k_rj}'=O(\alpha^{-1}\beta^{-1})$ if $r=2$.

3. Suppose that $i\in I_r$, $j\in I_t$, $r<t$. If neither $i=k_r$ nor $j=k_t$, then 
$$
c_{ij}'=\ell_{ij}^2-\ell_{k_rj}^2-\ell_{ik_t}^2+\ell_{k_rk_t}^2.
$$
By~\eqref{eq_IIsq}, $c_{ij}'=O(\alpha^{-2})$ if $(r,t)=(0,1)$ or $(2,3)
$, and $c_{ij}'=O(\alpha^{-1}\beta^{-1})$ if $(r,t)$ is one of the pairs $(0,2)$, $(0,3)$, $(1,2)$, and~$(1,3)$. If $i=k_r$ and $j\ne k_t$, then
$$
c_{k_rj}'=\ell_{k_rj}^2-\ell_{k_rk_t}^2-\ell_{0j}^2+\ell_{0k_t}^2=O(\alpha^{-1}\beta^{-1}).
$$
Now, suppose that $i\ne k_r$ and $j=k_t$. If $r=0$, and $t=1$ or~$2$, then 
$$
c'_{ik_t}=\ell_{ik_t}^2-\ell_{0k_t}^2-\ell_{0i}^2,
$$ 
which is $O(\alpha^{-2})$ for $t=1$ and $O(\alpha^{-1}\beta^{-1})$ for $t=2$. If $r=0$ or~$1$, and $t=3$, then
$$
c_{ik_3}'=\ell_{ik_3}^2-\ell_{k_rk_3}^2-\ell_{ik_2}^2+\ell_{k_rk_2}^2=O(\alpha^{-1}\beta^{-1}).
$$
If $r=1$ and $t=2$, then
$$
c_{ik_2}'=\ell_{ik_2}^2-\ell_{k_1k_2}^2-\ell_{0i}^2+\ell_{0k_1}^2=O(\alpha^{-1}\beta^{-1}).
$$
If $r=2$ and $t=3$, then
$$
c_{ik_3}'=\ell_{ik_3}^2-\ell_{k_2k_3}^2-\ell_{k_2i}^2=O(\alpha^{-2}).
$$
Finally,
\begin{align*}
c_{k_1k_2}'&=\ell_{k_1k_2}^2-\ell_{0k_1}^2-\ell_{0k_2}^2=2ab+O(\alpha^{-1}\beta^{-1})=O(\alpha\beta),\\
c_{k_1k_3}'&=\ell_{k_1k_3}^2-\ell_{0k_3}^2-\ell_{k_1k_2}^2+\ell_{0k_2}^2=-4ab-2b^2+O(\alpha^{-1}\beta^{-1})=O(\alpha\beta),\\
c_{k_2k_3}'&=\ell_{k_2k_3}^2+\ell_{0k_2}^2-\ell_{0k_3}^2=-2ab+O(\alpha^{-1}\beta^{-1})=O(\alpha\beta).
\end{align*}
The obtained estimates are summarized in Figure~\ref{fig_C'}.

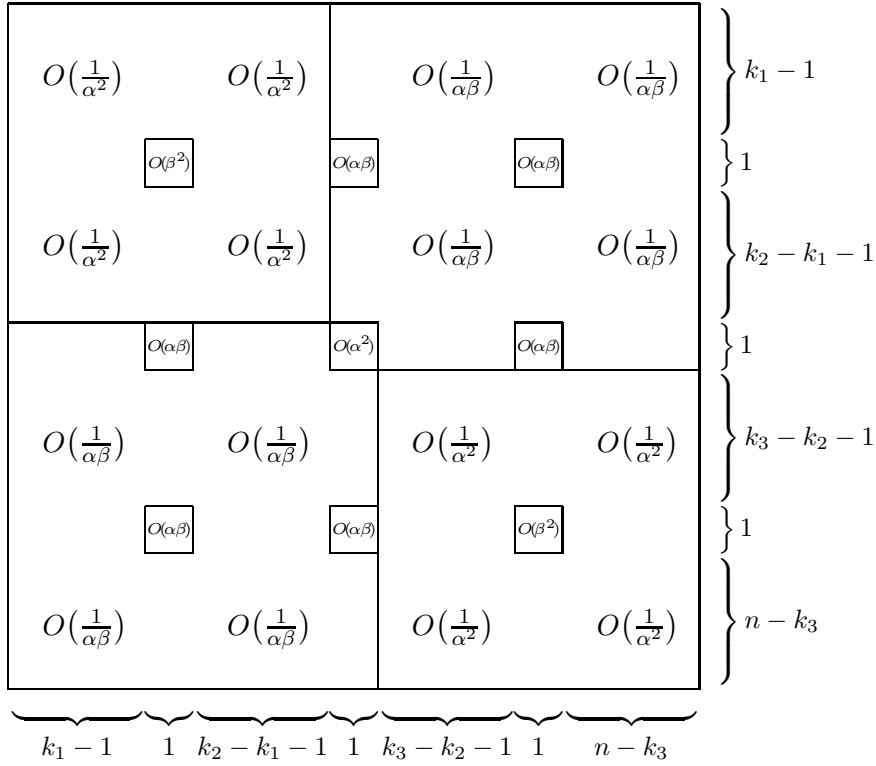
\begin{figure}
\begin{center}
\unitlength=.9cm
\begin{picture}(12.6,11.1)

\put(0,1){%
\begin{picture}(10.1,10.1)

\put(0,0){\line(1,0){10.1}}
\put(4.7,4.7){\line(1,0){5.4}}
\put(0,5.4){\line(1,0){5.4}}
\put(0,10.1){\line(1,0){10.1}}

\put(0,0){\line(0,1){10.1}}
\put(4.7,4.7){\line(0,1){5.4}}
\put(5.4,0){\line(0,1){5.4}}
\put(10.1,0){\line(0,1){10.1}}

\put(2,2){\line(1,0){.7}}
\put(2,2.7){\line(1,0){.7}}
\put(2.7,2){\line(0,1){.7}}
\put(2,2){\line(0,1){.7}}

\put(2,7.4){\line(1,0){.7}}
\put(2,8.1){\line(1,0){.7}}
\put(2.7,7.4){\line(0,1){.7}}
\put(2,7.4){\line(0,1){.7}}

\put(7.4,2){\line(1,0){.7}}
\put(7.4,2.7){\line(1,0){.7}}
\put(8.1,2){\line(0,1){.7}}
\put(7.4,2){\line(0,1){.7}}

\put(7.4,7.4){\line(1,0){.7}}
\put(7.4,8.1){\line(1,0){.7}}
\put(8.1,7.4){\line(0,1){.7}}
\put(7.4,7.4){\line(0,1){.7}}

\put(2,4.7){\line(1,0){.7}}
\put(2,4.7){\line(0,1){.7}}
\put(2.7,4.7){\line(0,1){.7}}

\put(4.7,2){\line(1,0){.7}}
\put(4.7,2){\line(0,1){.7}}
\put(4.7,2.7){\line(1,0){.7}}

\put(4.7,7.4){\line(1,0){.7}}
\put(4.7,8.1){\line(1,0){.7}}
\put(5.4,7.4){\line(0,1){.7}}

\put(7.4,4.7){\line(0,1){.7}}
\put(8.1,4.7){\line(0,1){.7}}
\put(7.4,5.4){\line(1,0){.7}}

\put(2.03,2.28){\tiny$O\!(\!\alpha\hspace{-.5pt}\beta\!)$}
\put(4.73,2.28){\tiny$O\!(\!\alpha\hspace{-.5pt}\beta\!)$}
\put(2.03,4.98){\tiny$O\!(\!\alpha\hspace{-.5pt}\beta\!)$}
\put(4.73,7.68){\tiny$O\!(\!\alpha\hspace{-.5pt}\beta\!)$}
\put(7.43,4.98){\tiny$O\!(\!\alpha\hspace{-.5pt}\beta\!)$}
\put(7.43,7.68){\tiny$O\!(\!\alpha\hspace{-.5pt}\beta\!)$}
\put(2.04,7.68){\tiny$O\!(\!\beta^2\!)$}
\put(7.44,2.28){\tiny$O\!(\!\beta^2\!)$}
\put(4.74,4.98){\tiny$O\!(\!\alpha\hspace{-.5pt}^2\!)$}

\put(.5,6.4){\large$O\bigl(\frac{1}{\alpha^2}\bigr)$}
\put(.5,8.9){\large$O\bigl(\frac{1}{\alpha^2}\bigr)$}
\put(3.2,6.4){\large$O\bigl(\frac{1}{\alpha^2}\bigr)$}
\put(3.2,8.9){\large$O\bigl(\frac{1}{\alpha^2}\bigr)$}

\put(5.9,6.4){\large$O\bigl(\frac{1}{\alpha\beta}\bigr)$}
\put(5.9,8.9){\large$O\bigl(\frac{1}{\alpha\beta}\bigr)$}
\put(8.6,6.4){\large$O\bigl(\frac{1}{\alpha\beta}\bigr)$}
\put(8.6,8.9){\large$O\bigl(\frac{1}{\alpha\beta}\bigr)$}

\put(5.9,.8){\large$O\bigl(\frac{1}{\alpha^2}\bigr)$}
\put(5.9,3.5){\large$O\bigl(\frac{1}{\alpha^2}\bigr)$}
\put(8.6,.8){\large$O\bigl(\frac{1}{\alpha^2}\bigr)$}
\put(8.6,3.5){\large$O\bigl(\frac{1}{\alpha^2}\bigr)$}

\put(.5,.8){\large$O\bigl(\frac{1}{\alpha\beta}\bigr)$}
\put(.5,3.5){\large$O\bigl(\frac{1}{\alpha\beta}\bigr)$}
\put(3.2,.8){\large$O\bigl(\frac{1}{\alpha\beta}\bigr)$}
\put(3.2,3.5){\large$O\bigl(\frac{1}{\alpha\beta}\bigr)$}

\put(10.3,.9){$\left.\rule{0pt}{10mm}\right\}n-k_3$}
\put(10.32,2.25){$\left.\rule{0pt}{3.3mm}\right\}1$}
\put(10.3,3.6){$\left.\rule{0pt}{10mm}\right\}k_3-k_2-1$}
\put(10.32,4.95){$\left.\rule{0pt}{3.3mm}\right\}1$}
\put(10.3,6.3){$\left.\rule{0pt}{10mm}\right\}k_2-k_1-1$}
\put(10.32,7.65){$\left.\rule{0pt}{3.3mm}\right\}1$}
\put(10.3,9){$\left.\rule{0pt}{10mm}\right\}k_1-1$}

\put(0.06,-.3){$\underbrace{\rule{17mm}{0pt}}_{\displaystyle\strut k_1-1}$}
\put(2,-.3){$\underbrace{\rule{5mm}{0pt}}_{\displaystyle\strut 1}$}
\put(2.76,-.3){$\underbrace{\rule{17mm}{0pt}}_{\displaystyle\strut k_2-k_1-1}$}
\put(4.7,-.3){$\underbrace{\rule{5mm}{0pt}}_{\displaystyle\strut 1}$}
\put(5.46,-.3){$\underbrace{\rule{17mm}{0pt}}_{\displaystyle\strut k_3-k_2-1}$}
\put(7.4,-.3){$\underbrace{\rule{5mm}{0pt}}_{\displaystyle\strut 1}$}
\put(8.16,-.3){$\underbrace{\rule{17mm}{0pt}}_{\displaystyle\strut n-k_3}$}

\end{picture}}
\end{picture}
\end{center}
\caption{The matrix $C'$}\label{fig_C'}

\end{figure}

Now, we divide the $k_1$th and the $k_3$th rows by~$ab$, and divide the $k_2$th row by~$a^2$. Further, we divide  the $k_1$th and the $k_3$th columns by~$ab$, and divide the $k_2$th column by~$a^2$. Let~$C''$ be the matrix obtained. Then all elements of~$C''$ are~$O(\alpha^{-1}\beta^{-1})$. Besides, the upper-left $k_2\times k_2$ submatrix of~$C''$ and the below-right $(n-k_2+1)\times(n-k_2+1)$ submatrix of~$C''$ consist of elements that are~$O(\alpha^{-2})$.
Hence, 
\begin{align*}
v(\det C'')&\le \alpha^{-2}\cdot\left(\alpha^{-1}\beta^{-1}\right)^{n-1}= 
\alpha^{-n-1}\beta^{-n+1}&&\text{if $n$ is odd,}\\ 
v(\det C'')&\le \left(\alpha^{-2}\right)^2\cdot\left(\alpha^{-1}\beta^{-1}\right)^{n-2}= 
\alpha^{-n-2}\beta^{-n+2}&&
\text{if $n$ is even.}
\end{align*}
Assertions~(i) and~(ii) follow, since $v(V)^2=v(\det C')=\alpha^8\beta^4v(\det C'')$.

Let us prove that $v(P_{i_1i_2i_3i_4i_5})\le \beta^4$. Certainly, this is true  if $\beta>1$. If $\beta\le 1$, then $\beta<\alpha$. Hence, the set~$I$ consists of two $O(\beta)$-components~$I_0\cup I_1$ and~$I_2\cup I_3$. Three of the five numbers $i_1,\ldots,i_5$, say  $i_1$, $i_2$, and~$i_3$, belong to the same $O(\beta)$-component. Then $v(S_{i_1i_2i_3})\le\beta^4$, hence, $v(P_{i_1i_2i_3i_4i_5})\le\beta^4$.

 For even~$n$, (iii) follows from~(i). 
Suppose that $n$ is odd, $n\ge 7$. If $\beta^2\ge \alpha^{-1}$ and $n\ge 9$, then 
$
v(V)\le \alpha^{\frac{7-n}{2}}\beta^{\frac{5-n}{2}}\le \alpha^{\frac{9-n}{4}}\le 1.
$
Similarly, if $n=7$ and $\beta\ge 1$, then
$
v(V)\le\beta^{-1}\le 1
$.
Hence in both cases $v(P_{i_1i_2i_3i_4i_5}V^2)\le 1$. 

If $\beta^2<\alpha^{-1}$,  then $v(P_{i_1i_2i_3i_4i_5})<\alpha^{-2}$. Hence, $
v(P_{i_1i_2i_3i_4i_5}V^2)\le(\alpha\beta)^{5-n}<1.
$

If $n=7$ and $\beta<1$, we obtain that
$
v(P_{i_1i_2i_3i_4i_5}V^2)\le\beta^4\cdot\beta^{-2}=\beta^2<1.
$

Let us prove~(iv). By~(ii), we have $v(V)\le\alpha$. Since $|I|=6$, we see that either one of the subsets~$I_r$ has cardinality~$3$, or two subsets~$I_r$ and~$I_t$ have cardinalities~$2$.

Suppose that $I_r=\{i,j,k\}$ has cardinality~$3$. Then $v(\ell_{ij})$, $v(\ell_{jk})$, and~$v(\ell_{ki})$ do not exceed~$\alpha^{-1}$. Hence, $v(S_{ijk})\le\alpha^{-4}$. Therefore, $v(P_{i_1i_2i_3i_4i_5})\le\alpha^{-4}$. Thus,  $v(P_{i_1i_2i_3i_4i_5}DV^2)\le\alpha^{-2}<1$.

Suppose that $I_r=\{i,j\}$ and $I_t=\{k,l\}$ have cardinalities~$2$. Then $\ell_{ij}=O(\alpha^{-1})$, and the difference between any two of the elements~$\varepsilon_{ik}\ell_{ik}$, $\varepsilon_{jk}\ell_{jk}$, $\varepsilon_{il}\ell_{il}$, and~$\varepsilon_{jl}\ell_{jl}$ is~$O(\alpha^{-2}\beta^{-1})$. By Heron's formula we easily get $S_{ijk}=\frac{1}{4}\ell_{ik}^2\ell_{ij}^2+O(\alpha^{-2}\beta^{-2})$ and $S_{ijl}=\frac{1}{4}\ell_{il}^2\ell_{ij}^2+O(\alpha^{-2}\beta^{-2})$. Therefore, $S_{ijk}-S_{ijl}=O(\alpha^{-2}\beta^{-2})$. Similarly, $S_{ikl}-S_{jkl}=O(\alpha^{-2}\beta^{-2})$. Consequently, $v(D)\le\alpha^{-4}\beta^{-4}$. Thus, $$v(P_{i_1i_2i_3i_4i_5}DV^2)\le \beta^4\cdot(\alpha^{-4}\beta^{-4})\cdot\alpha^2= \alpha^{-2}<1.$$

Let us prove~(v).
The set $I=\{0,1,2,3,4\}$ is the disjoint union of the non-empty subsets~$I_0$, $I_1$, $I_2$, and~$I_3$. Without loss of generality we may assume that $I_0=\{0,1\}$, $I_1=\{2\}$, $I_2=\{3\}$, and $I_3=\{4\}$.  Then $\ell_{01}=O(\alpha^{-1})$, and
\begin{equation*}
\ell_{i2}=\pm b+O(\alpha^{-2}\beta^{-1}),\quad
\ell_{i3}=\pm a+O(\alpha^{-2}\beta^{-1}),\quad
\ell_{i4}=\pm (a+b)+O(\alpha^{-2}\beta^{-1})
\end{equation*}
for $i=0,1$.
By Heron's formula, we easily obtain that
\begin{gather*}
A_{012}=\pm\frac12b\ell_{01}+O(\alpha^{-3}\beta^{-1}),\qquad
A_{013}=\pm\frac12a\ell_{01}+O(\alpha^{-2}\beta^{-2}),\\
A_{014}=\pm\frac12(a+b)\ell_{01}+O(\alpha^{-2}\beta^{-2}).
\end{gather*}
Hence, 
$
\pm A_{012}\pm A_{013}\pm A_{014}=O(\alpha^{-2}\beta^{-2})
$ for some choice of signs. Therefore,
 $v(Q)\le \alpha^{-2}\beta^{-2}$. But by~(i), $v(V)\le\alpha\beta$. Thus, $v(QV^2)\le 1$.
\end{proof}

\section{Proofs of Theorems~\ref{theorem_R}, \ref{theorem_Lambda}, \ref{theorem_leading}, and~\ref{theorem_leading2}}\label{section_final}

Suppose that $n\ge 6$ and $n$ is even. To prove that $V$ is integral over~$R_2$, we need to show that $v(V)\le 1$ for all valuations~$v\in \fV$. For $v\in\fV\setminus\fV_>$, this follows from the Cayley--Menger formula. Any valuation $v\in\fV_>$ is either of type~I or of type~II. By Lemmas~\ref{lem_typeI}(i) and~\ref{lem_typeII}(i), $v(V)\le 1$. Thus, $V$ is integral over~$R_2$.

If $n\ge 7$ and $n$ is odd, Lemmas~\ref{lem_typeI}(v) and~\ref{lem_typeII}(iii) imply that $v(P_{i_1i_2i_3i_4i_5}V^2)\le 1$ for any $v\in\fV$. Hence, $P_{i_1i_2i_3i_4i_5}V^2$ is integral over~$R_2$. Since $P_{i_1i_2i_3i_4i_5}$ is a monomial in~$A_{ijk}$, we obtain that $V$ is integral over~$\Lambda_2$. Similarly, for $n=5$, Lemmas~\ref{lem_typeI}(v) and~\ref{lem_typeII}(iv) imply that  $P_{i_1i_2i_3i_4i_5}DV^2$ is integral over~$R_2$. 
In section~\ref{section_negative} we have proved that $V$ is not integral over~$R_2$ if $n$ is odd, and is not integral over~$\Lambda_2$ if $n=4$ or~$5$. Thus, we have proved Theorems~\ref{theorem_R}, \ref{theorem_Lambda}, and~\ref{theorem_leading2}.

\begin{proof}[Proof of Theorem~\ref{theorem_leading}] 
Lemmas~\ref{lem_typeI}(i) and~\ref{lem_typeII}(v) imply that $QW$ is integral over~$R_2$. Hence, $QW$ is integral over~$R_2'=\Q[\bS_2]$. Therefore, $Q\in\mathfrak{b}$. The proof of the following algebraic lemma is standard.

\begin{lem}
Let $R$ be a factorial domain contained in a field~$F$, and let $y\in F$ be an element that is algebraic over~$R$, but not integral over~$R$. Then the ideal\/ $\mathfrak{c}\lhd R$ consisting of all\/ $x\in R$ such that $xy$ is integral over~$R$ is principal.
\end{lem}

The ring~$\Q[\bS_2]$ is a factorial domain. Therefore, the ideal~$\mathfrak{b}$ is principal. Suppose, $\mathfrak{b}=(B)$, then $B$ is a divisor of~$Q$.  The symmetric group~$\mathfrak{S}_5$  acts naturally on~$\Q[\bell]$ by~$\nu\cdot\ell_{ij}=\ell_{\nu(i)\nu(j)}$. The subring $\Q[\bS_2]$ is invariant under this action, since $\nu\cdot S_{ijk}=S_{\nu(i)\nu(j)\nu(k)}$. The element $W$ is fixed by the action of~$\mathfrak{S}_5$. Hence the ideal~$\mathfrak{b}$ is $\mathfrak{S}_5$-invariant. Therefore, $\nu\cdot B=\pm B$ for all $\nu\in\mathfrak{S}_5$. However, it is easy to check that the polynomial~$Q$ has no proper divisors in~$\Q[\bS_2]$  invariant under the action of~$\mathfrak{S}_5$ up to sign. Thus, $B=Q$.
\end{proof}

\end{document}